\documentclass[12pt]{article}
\usepackage[]{amsmath,amssymb}
\usepackage{amscd}
\usepackage{latexsym}
\usepackage{cite}
\usepackage{amsthm}

\newtheorem{definition}{Definition}[section]
\newtheorem{theorem}[definition]{Theorem}
\newtheorem{lemma}[definition]{Lemma}
\newtheorem{corollary}[definition]{Corollary}

\newtheorem{note}[definition]{Note}

\newtheorem{proposition}[definition]{Proposition}

\begin{document}
\title{\bf  
The Lusztig automorphism of 
$U_q(\mathfrak{sl}_2)$  \\
from the equitable point of view
}
\author{
Paul Terwilliger}
\date{}

\maketitle
\begin{abstract}
We consider the 
quantum algebra 
$U_q(\mathfrak{sl}_2)$ in the equitable presentation.
From this point of view,
 we describe the Lusztig automorphism
and the corresponding Lusztig operator.
\bigskip

 \end{abstract}
\section{Introduction}
This paper is about the quantum algebra
 $U_q(\mathfrak {sl}_2)$, but in order to motivate things
we have some preliminary comments about the Lie algebra
 $\mathfrak {sl}_2$.   Let 
 $\mathbb F$ denote a field with characteristic zero, and
 consider the Lie algebra
 $\mathfrak {sl}_2$ over $\mathbb F$.
The following facts
  are taken from
\cite[Sections~2.3,~7]{humphreys}. 
The Lie algebra
 $\mathfrak {sl}_2$ 
 has a basis $e,f,h$ and Lie bracket
$$
\lbrack h,e\rbrack = 2e, \qquad \qquad  
\lbrack h,f\rbrack = -2f,  \qquad \qquad 
\lbrack e,f\rbrack = h.
$$
On each finite-dimensional 
 $\mathfrak {sl}_2$-module, $h$ is diagonalizable and $e,f$
 are nilpotent.
The 
 Lie algebra $\mathfrak {sl}_2$ has an automorphism
\begin{eqnarray}
\label{eq:sl2L}
{\mathcal L} = 
{\rm exp}({\rm ad}\, e)\, 
{\rm exp}(-{\rm ad}\, f)\, 
{\rm exp}({\rm ad} \,e), 
\end{eqnarray}
where 
${\rm ad}\, u \, (v) = \lbrack u,v\rbrack$  and
${\rm exp}(\varphi) =\sum_{i\in \mathbb N} \varphi^i /i!$.
The automorphism $\mathcal L$ sends
\begin{eqnarray*}
e\mapsto -f, \qquad  \qquad 
f\mapsto -e, \qquad \qquad
h\mapsto -h.
\end{eqnarray*}
The operator
\begin{eqnarray}
\label{eq:sl2T}
{\mathcal T} = 
{\rm exp}(e)\, 
{\rm exp}(-f)\, 
{\rm exp}(e) 
\end{eqnarray}
 acts on nonzero finite-dimensional 
 $\mathfrak {sl}_2$-modules.
On these modules,
\begin{eqnarray*}
{\mathcal L}(\xi) = {\mathcal T} \xi {\mathcal  T}^{-1} \qquad \qquad
\forall \; \xi \in 
\mathfrak {sl}_2.
\end{eqnarray*}
We are done discussing 
 $\mathfrak {sl}_2$. We now turn our attention to 
the analog 
situation
 for $U_q(\mathfrak {sl}_2)$.
From now on, let the field $\mathbb F$ be
arbitrary. Fix a nonzero $q \in \mathbb F$ that
is not a root of 1, and consider the
algebra
  $U_q(\mathfrak {sl}_2)$ over $\mathbb F$.
By 
\cite[Definition~1.1]{jantzen},
 the Chevalley presentation of
 $U_q(\mathfrak {sl}_2)$ has generators
$e,f,k^{\pm 1}$ and relations $kk^{-1} =1$, $k^{-1}k=1$,
\begin{eqnarray*}
ke = q^2 ek, \qquad \qquad 
kf = q^{-2} fk, \qquad \qquad 
ef-fe = \frac{k-k^{-1}}{q-q^{-1}}.
\end{eqnarray*}
The Lusztig automorphism $L$ of
 $U_q(\mathfrak {sl}_2)$ was introduced in
\cite{lusztig1,lusztig2}  
 as a quantum analog of
(\ref{eq:sl2L}). By 
\cite[Section~8.14]{jantzen},
$L$  sends
\begin{eqnarray*}
e\mapsto -fk, \qquad \qquad
f\mapsto -k^{-1}e, \qquad \qquad
k\mapsto k^{-1}.
\end{eqnarray*}
The Lusztig operator $T$ was introduced in
\cite{lusztig1,lusztig2}  
as a quantum analog of
(\ref{eq:sl2T}). $T$ acts on a family of finite-dimensional
 $U_q(\mathfrak {sl}_2)$-modules, said to have type 1.
This family is described as follows.
A nonzero finite-dimensional
 $U_q(\mathfrak {sl}_2)$-module $V$
has type 1
if and only if $k$ is diagonalizable on $V$, with
all eigenvalues among 
$\lbrace q^\lambda \rbrace_{\lambda \in \mathbb Z}$.
Assume that $V$ has type 1. For $\lambda \in \mathbb Z$
the corresponding weight space 
$V(\lambda)= \lbrace v \in V| kv=q^\lambda v \rbrace$.
By 
\cite[Section~8.2]{jantzen}
we find that on each weight space $V(\lambda)$,
\begin{eqnarray*}
T = 
\sum_{
\genfrac{}{}{0pt}{}{a,b,c \in \mathbb N}{b-a-c=\lambda}
}
\frac{e^a f^b e^c}{
\lbrack a \rbrack^!_q
\lbrack b \rbrack^!_q
\lbrack c \rbrack^!_q}
(-1)^b q^{b-ac}.
\end{eqnarray*}
(The bracket notation is explained in Section 2.)
By \cite[Lemmas~8.4,~8.5]{jantzen}, on $V$
the operator $T$ is invertible and
\begin{eqnarray*}
L(\xi) = T \xi T^{-1} \qquad \qquad
\forall \; \xi \in 
U_q(\mathfrak {sl}_2).
\end{eqnarray*}
The equitable presentation of
 $U_q(\mathfrak {sl}_2)$ was introduced in
\cite{equit}
and investigated further in 
\cite{
alnajjar,
bocktingTer,
bidiag,
irt,
qtet,
nomura,
tersym,
uawe,
fduq,
billiard,
lrt,
wora
}. 
This presentation has generators
$x,y^{\pm 1},z$ and relations
$y y^{-1} = 1$, 
$y^{-1} y = 1$,
\begin{eqnarray*}
\frac{q xy -q^{-1} yx}{q-q^{-1}} = 1, 
\qquad
\frac{q yz -q^{-1} zy}{q-q^{-1}} = 1, 
\qquad 
\frac{q zx -q^{-1} xz}{q-q^{-1}} = 1.
\end{eqnarray*}
In Section 5 we explain how  $x,y,z$ are related to $e,f,k$.
This relationship is not unique. From the set of 
possible relationships we select two that seem
attractive (see Corollary
\ref{cor:goodtheta});
these exist provided that
 $q^{1/2} \in \mathbb F$. 
For the rest of this section, assume that $q^{1/2} \in \mathbb F$
and adopt one of the two selections.
\medskip

\noindent 
Our goal for this paper is to describe how $L$ and $T$ look
in the equitable presentation. We now describe $L$.
Define
\begin{eqnarray*}
n_x = \frac{q(1-yz)}{q-q^{-1}},
\qquad \qquad 
n_y = \frac{q(1-zx)}{q-q^{-1}},
\qquad \qquad 
n_z = \frac{q(1-xy)}{q-q^{-1}}.
\end{eqnarray*}
By \cite[Lemma~5.4]{equit},
\begin{eqnarray*}
&&x n_y = q^2 n_y x,
\qquad \qquad y n_z = q^2 n_z y,
\qquad \qquad z n_x = q^2 n_x z, 
\\
&&
x n_z = q^{-2} n_z x,
 \qquad \qquad y n_x = q^{-2} n_x y,
 \qquad \qquad z n_y = q^{-2} n_y z.
\end{eqnarray*}
On nonzero finite-dimensional
 $U_q(\mathfrak {sl}_2)$-modules,
each of $x,y,z$ is invertible (see Lemma 
\ref{lem:xyzinv})
and 
each of $n_x,n_y,n_z$ is nilpotent (see Lemma 
\ref{lem:nxNIL}.)
By 
\cite[Lemma~6.4]{uawe},
the algebra
 $U_q(\mathfrak {sl}_2)$ is generated by
$n_x,y^{\pm 1},n_z$.
As we will see in Corollary
\ref{cor:goodtheta},
$L$ sends
\begin{eqnarray*}
n_x \mapsto  y^{-1} n_z y^{-1},
\qquad \qquad 
y \mapsto y^{-1},
\qquad \qquad 
n_z \mapsto  n_x.
\end{eqnarray*}
We now describe $T$.
We define an operator $\Upsilon $  
that acts on finite-dimensional 
 $U_q(\mathfrak {sl}_2)$-modules of type 1.
Let $V$ denote such a module.
On each weight space $V(\lambda)$, $\Upsilon$ acts 
as 
$q^{-\lambda^2/2}$ times the identity.
We recall the notion of a rotator.
This notion is implicit in \cite{equit} 
and explicit in \cite[Section~16]{fduq};  see also 
\cite[Section~22]{lrt}.
Let $V$ denote a finite-dimensional 
 $U_q(\mathfrak {sl}_2)$-module of type 1. A rotator on $V$ 
 is an invertible $R \in {\rm End}(V)$ such that on $V$,
\begin{eqnarray*}
R^{-1} x R = y,
\qquad \qquad
R^{-1} y R = z,
\qquad \qquad
R^{-1} z R = x.
\end{eqnarray*}
Recall the $q$-exponential function 
\begin{eqnarray*}
{\rm exp}_q(\varphi) = \sum_{i\in \mathbb N} \frac{q^{\binom{i}{2}}}
{\lbrack i \rbrack^!_q}
\varphi^i.
\end{eqnarray*}
As we will see in Lemma 
\ref{lem:RisRot},
the operator
\begin{eqnarray*}
\mathfrak R = 
{\rm exp}_q(n_x)\,
\Upsilon \,
{\rm exp}_q(n_z)
\end{eqnarray*}
acts as a rotator on each finite-dimensional 
 $U_q(\mathfrak {sl}_2)$-module of type 1.
We can now easily describe $T$ in the
equitable presentation. In Theorem 
\ref{thm:LOm}
we show that
on each finite-dimensional
 $U_q(\mathfrak {sl}_2)$-module of type 1,
 \begin{eqnarray*}
T^{-1} = {\rm exp}_q(n_z)\, \mathfrak{R}.
\end{eqnarray*}
We have been describing the Lusztig automorphism $L$ and
the Lusztig operator $T$. For both maps there is a second version,
which we denote by $L^\vee$ 
and $T^\vee$, respectively. The maps
$L^\vee,T^\vee$ are treated along with
$L,T$ in the main body of the paper.
\medskip

\noindent The paper is organized as follows.
Section 2 contains some preliminaries.
Section 3 contains some basic facts about the
Chevalley presentation of
 $U_q(\mathfrak {sl}_2)$.
In Section 4 we discuss how the Lusztig automorphisms
$L, L^\vee $
and the Lusztig operators $T,T^\vee$ look in the
Chevalley presentation.
In Section 5 we discuss the
equitable presentation of 
 $U_q(\mathfrak {sl}_2)$, and describe how
 $L,L^\vee$ look in this presentation.
In Section 6 we review the $q$-exponential function,
and apply it to $n_x, n_y, n_z$. Section 7 is about
rotators. In Sections 8, 9 we use rotators to describe how the Lusztig
operators $T,T^\vee$ look in the equitable presentation.
Theorem 
\ref{thm:LOm}
is the main
result of the paper.

\section{Preliminaries}

\noindent
We now begin our formal argument.
Throughout the paper the following
notation and assumptions are in effect.
An algebra without the Lie prefix 
is meant to be associative and have a 1.
 Recall the natural numbers
$\mathbb N=\lbrace 0,1,2,\ldots \rbrace$ and
integers
$\mathbb Z = \lbrace 0, \pm 1,\pm2, \ldots \rbrace $.
Let $\mathbb F$ denote a field,
and let $V$ denote a vector space over $\mathbb F$
with finite positive dimension.
Let ${\rm End}(V)$ denote the $\mathbb F$-algebra
consisting of the $\mathbb F$-linear maps
from $V$ to $V$. 
An element $\varphi \in {\rm End}(V)$  
is called {\it diagonalizable} whenever 
$V$ is spanned by the eigenspaces of $\varphi$.
 The map $\varphi $ is called
{\it multiplicity-free} whenever $\varphi $ is
diagonalizable, and each eigenspace has dimension
one. The map $\varphi $ is called {\it nilpotent}
whenever there exists a positive integer $n$ such that
$\varphi^n=0$.
Fix a nonzero $q \in \mathbb F$ that is not a root of 1.
For $n \in \mathbb Z$  define
\begin{eqnarray*}
\lbrack n \rbrack_q = \frac{q^n-q^{-n}}{q-q^{-1}},
\end{eqnarray*}
and for $n \geq 0$  define
\begin{eqnarray*}
\lbrack n \rbrack^!_q = 
\lbrack n \rbrack_q
\lbrack n-1 \rbrack_q
\cdots \lbrack 2 \rbrack_q
\lbrack 1 \rbrack_q.
\end{eqnarray*}
We interpret 
$\lbrack 0 \rbrack^!_q = 1$.

\section{The Chevalley presention of
 $U_q(\mathfrak {sl}_2)$
}

\noindent We recall the quantum algebra
 $U_q(\mathfrak {sl}_2)$ in the Chevalley presentation,
 following the treatment in
\cite{jantzen}.

\begin{definition}
\label{def:chev}
\rm
(See \cite[Definition~1.1]{jantzen}.)
Let $U_q(\mathfrak {sl}_2)$
denote the $\mathbb F$-algebra with generators
$e,f,k^{\pm 1}$ and relations
\begin{eqnarray*}
&&
k k^{-1} =1, \qquad \qquad \; k^{-1}k = 1,
\\
&&
ke=q^2 ek, \qquad \qquad kf = q^{-2} fk,
\\
&&
ef-fe = \frac{k-k^{-1}}{q-q^{-1}}.
\end{eqnarray*}
The elements $e,f,k^{\pm 1}$ are called
the {\it Chevalley generators} for
$U_q(\mathfrak {sl}_2)$.
\end{definition}

\noindent Define
\begin{eqnarray}
\Lambda = (q-q^{-1})^2 ef + q^{-1} k + q k^{-1}.
\label{eq:lambda}
\end{eqnarray}
We call $\Lambda$ the {\it normalized Casimir element} of 
$U_q(\mathfrak {sl}_2)$.
The element $\Lambda (q-q^{-1})^{-2}$
is the Casimir element of
$U_q(\mathfrak {sl}_2)$ discussed in
\cite[Section~2.7]{jantzen}.
By 
\cite[Proposition~2.18]{jantzen} 
the elements
$\lbrace \Lambda^n \rbrace_{n \in \mathbb N}$
form a basis for the center of
$U_q(\mathfrak {sl}_2)$.
\medskip

\noindent 
We now recall the finite-dimensional 
irreducible $U_q(\mathfrak {sl}_2)$-modules.



\begin{lemma}
\label{lem:irred}
{\rm (See \cite[Theorem~2.6]{jantzen}.)}
There exists a family of finite-dimensional
irreducible 
$U_q(\mathfrak {sl}_2)$-modules
\begin{eqnarray}
{\bf V}_{d,\varepsilon} \qquad \qquad d \in \mathbb N,
\qquad \qquad 
\varepsilon \in \lbrace 1,-1 \rbrace
\label{eq:Lde}
\end{eqnarray}
with the following properties: 
${\bf V}_{d,\varepsilon}$ 
has a basis $\lbrace v_i \rbrace_{i=0}^d$ such that
\begin{eqnarray*}
&&k v_i = \varepsilon q^{d-2i} v_i \qquad (0 \leq i \leq d),
\\
&& f v_i = \lbrack i+1 \rbrack_q v_{i+1} \qquad (0 \leq i \leq d-1),
\qquad f v_d = 0,
\\
&& e v_i = \varepsilon \lbrack d-i+1 \rbrack_q v_{i-1} \qquad
(1 \leq i \leq d),
\qquad e v_0 = 0.
\end{eqnarray*}
Each finite-dimensional irreducible
$U_q(\mathfrak {sl}_2)$-module is isomorphic to exactly one
of the modules 
from line {\rm (\ref{eq:Lde})}.
\end{lemma}

\noindent We have some comments about Lemma
\ref{lem:irred}.
 The
$U_q(\mathfrak {sl}_2)$-module ${\bf V}_{d,\varepsilon}$
has dimension $d+1$.
Referring to line
(\ref{eq:Lde}),
if ${\rm Char}(\mathbb F)=2$  then we view
$\lbrace 1,-1\rbrace$ as containing a single element.
On the 
$U_q(\mathfrak {sl}_2)$-module 
${\bf V}_{d,\varepsilon}$ the element
$k$ is multiplicity-free, with eigenvalues
$\lbrace \varepsilon q^{d-2i} \;|\; 0 \leq i \leq d \rbrace$.
 Moreover each of $e^r,f^r$ is nonzero for
$0 \leq r \leq d$, and 
\begin{eqnarray*}
e^{d+1} = 0, \qquad  \qquad f^{d+1} = 0.
\end{eqnarray*}
By \cite[Lemma~2.7]{jantzen}
the normalized Casimir element
$\Lambda $ acts on
${\bf V}_{d,\varepsilon}$  as
$\varepsilon (q^{d+1}+q^{-d-1})I$.

\begin{definition}
\label{def:VD}
\rm
Referring to the 
$U_q(\mathfrak {sl}_2)$-module 
${\bf V}_{d,\varepsilon}$ 
from
Lemma
\ref{lem:irred}, we call
$d$ and 
$\varepsilon $
the {\it diameter} and
{\it type}, respectively.
 We often abbreviate
${\bf V}_d ={\bf V}_{d,1}$. 
\end{definition}


\noindent Next we consider finite-dimensional
$U_q(\mathfrak {sl}_2)$-modules that are not
necessarily irreducible.


\begin{lemma}
{\rm (See \cite[Proposition~2.1]{jantzen}.)}
The elements $e$ and $f$ are nilpotent
on every finite-dimensional 
$U_q(\mathfrak {sl}_2)$-module.
\end{lemma} 

\begin{lemma}
\label{lem:CR}
{\rm (See \cite[Proposition~2.3, Theorem~2.9]{jantzen}.)}
Let $V$ denote a finite-dimensional
$U_q(\mathfrak {sl}_2)$-module. Then the following
are equivalent:
\begin{enumerate}
\item[\rm (i)] 
 the action of $k$ on
$V$ 
is diagonalizable;
\item[\rm (ii)] 
$V$ is a direct sum of  irreducible
$U_q(\mathfrak {sl}_2)$-submodules.
\end{enumerate}
Moreover, if
 ${\rm Char}(\mathbb F)\not=2$ then
the conditions {\rm (i), (ii)} hold.
\end{lemma}

\begin{definition}
\label{def:typeONE}
\rm
Let $V$ denote a finite-dimensional 
$U_q(\mathfrak {sl}_2)$-module. Then $V$ is said to have
{\it type 1} whenever $V\not=0$ and $V$ is a direct sum of
irreducible
$U_q(\mathfrak {sl}_2)$-submodules that have type 1.
\end{definition}

\noindent From our above comments we routinely obtain the following
result.

\begin{lemma} 
\label{lem:kWS}
Let $V$ denote a
nonzero finite-dimensional $U_q(\mathfrak {sl}_2)$-module.
Then the following are equivalent:
\begin{enumerate}
\item[\rm (i)] $V$ has type 1;
\item[\rm (ii)] 
 the action of $k$ on
$V$ 
is diagonalizable, and each eigenvalue is among
$\lbrace q^{\lambda} \rbrace_{\lambda \in \mathbb Z}$.
\end{enumerate}
\end{lemma}

\begin{lemma}
\label{lem:typeOneSub}
Let $V$ denote a finite-dimensional
$U_q(\mathfrak {sl}_2)$-module of type 1.
Then each nonzero 
$U_q(\mathfrak {sl}_2)$-submodule of $V$ has type 1.
\end{lemma}
\begin{proof} Apply  
Lemma \ref{lem:kWS} to the 
$U_q(\mathfrak {sl}_2)$-submodule in question.
\end{proof}

\begin{definition}
\label{def:WS}
\rm
Let $V$ denote a finite-dimensional
$U_q(\mathfrak {sl}_2)$-module of type 1.
For  $\lambda \in \mathbb Z$ 
define
\begin{eqnarray*}
V(\lambda) = \lbrace v \in V \;|\; kv= 
q^\lambda v\rbrace.
\end{eqnarray*}
We call 
 $V(\lambda)$ the
{\it $\lambda$-weight space} of $V$.
Note that $V(\lambda) \not=0$ if and only if
$q^\lambda$
is an eigenvalue of $k$ on $V$, and in  this case
$V(\lambda)$ is the corresponding eigenspace.
\end{definition}

\begin{lemma}
\label{lem:efmove}
{\rm (See \cite[Section~2.2]{jantzen}.)}
Referring to Definition
\ref{def:WS}, for each weight space $V(\lambda)$ we have
\begin{eqnarray*}
e V(\lambda)
\subseteq 
 V(\lambda+2),
\qquad \qquad 
f V(\lambda)
\subseteq 
 V(\lambda-2).
\end{eqnarray*}
\end{lemma}

\begin{lemma} 
\label{lem:dsws}
Each finite-dimensional 
$U_q(\mathfrak {sl}_2)$-module of type 1 is a direct sum of its
weight spaces.
\end{lemma}
\begin{proof}
By Lemma \ref{lem:kWS} and
Definition
\ref{def:WS}.
\end{proof}

\noindent By an {\it automorphism} of 
$U_q(\mathfrak {sl}_2)$ we mean an $\mathbb F$-algebra
isomorphism from
$U_q(\mathfrak {sl}_2)$ to 
$U_q(\mathfrak {sl}_2)$.

\section{The Lusztig automorphism of 
$U_q(\mathfrak {sl}_2)$}

\noindent In this section we recall the Lusztig automorphism of
$U_q(\mathfrak {sl}_2)$. Our treatment 
follows 
\cite[Chapter~8]{jantzen}, more or less.
The Lusztig automorphism has two versions, which we now define.

\begin{definition}
\rm (See \cite[Section~8.14]{jantzen}.)
\label{def:LusztigAut}
\rm The {\it Lusztig automorphism} $L$ of
$U_q(\mathfrak {sl}_2)$ sends
\begin{eqnarray*}
e\mapsto -fk, \qquad \qquad
f\mapsto -k^{-1}e, \qquad \qquad
k\mapsto k^{-1}.
\end{eqnarray*}
\rm The {\it Lusztig automorphism} $L^\vee$ of
$U_q(\mathfrak {sl}_2)$ sends
\begin{eqnarray*}
e\mapsto -kf, \qquad \qquad
f\mapsto -ek^{-1}, \qquad \qquad
k\mapsto k^{-1}.
\end{eqnarray*}
\end{definition}

\noindent The inverses of
$L$ and $L^\vee$ are described as follows.
\begin{lemma} 
\label{lem:Linv}
The automorphism $L^{-1}$ sends
\begin{eqnarray*}
e\mapsto -k^{-1}f, \qquad \qquad
f\mapsto -ek, \qquad \qquad
k\mapsto k^{-1}.
\end{eqnarray*}
The automorphism $(L^\vee)^{-1}$ sends
\begin{eqnarray*}
e\mapsto -fk^{-1}, \qquad \qquad
f\mapsto -ke, \qquad \qquad
k\mapsto k^{-1}.
\end{eqnarray*}
\end{lemma}
\begin{proof} Routine using Definition
\ref{def:LusztigAut}.
\end{proof}

\noindent We now explain how $L$ and $L^\vee$ are related.
\begin{lemma}
\label{lem:CD}
The following diagram commutes: 
\begin{equation*}
\begin{CD}
U_q(\mathfrak {sl}_2)
           @> I >>
                      U_q(\mathfrak {sl}_2) 
	   \\ 
          @VL VV                   @VVL^{\vee}V \\
                      U_q(\mathfrak {sl}_2) 
                              @>>u \mapsto kuk^{-1}>   U_q(\mathfrak {sl}_2) 
                   \end{CD}
\end{equation*}
\end{lemma}
\begin{proof}  Chase the 
                      $U_q(\mathfrak {sl}_2)$ generators 
$e,f,k^{\pm 1}$ around the diagram using Definition
\ref{def:LusztigAut}.
\end{proof}

\noindent Consider the Lusztig automorphisms $L$ and 
$L^\vee$.
Associated with $L$ (resp. $L^{\vee}$) 
is a certain operator $T$ (resp. $T^{\vee}$) called its Lusztig operator, that
acts on
each type 1 finite-dimensional 
$U_q(\mathfrak {sl}_2)$-module in an $\mathbb F$-linear
fashion.
We now give the action.

\begin{definition}
\label{def:lusztigOp}
\rm (See \cite[Section~8.2]{jantzen}.)
The Lusztig operators $T$ and $T^{\vee}$ act as follows.
Let $V$ denote a finite-dimensional 
$U_q(\mathfrak {sl}_2)$-module of type 1. On each weight space
$V(\lambda)$,
\begin{eqnarray*}
&&T = 
\sum_{
\genfrac{}{}{0pt}{}{a,b,c \in \mathbb N}{b-a-c=\lambda}
}
\frac{e^a f^b e^c}{
\lbrack a \rbrack^!_q
\lbrack b \rbrack^!_q
\lbrack c \rbrack^!_q}
(-1)^b q^{b-ac},
\\
&&
T^{\vee} = 
\sum_{
\genfrac{}{}{0pt}{}{a,b,c \in \mathbb N}{a-b+c=\lambda}
}
\frac{f^a e^b f^c}{
\lbrack a \rbrack^!_q
\lbrack b \rbrack^!_q
\lbrack c \rbrack^!_q}
(-1)^b q^{b-ac}.
\end{eqnarray*}
\end{definition}

\begin{lemma}
\label{lem:TTcheckInv}
{\rm (See \cite[Lemma~8.4]{jantzen}.)}
Let $V$ denote a finite-dimensional 
$U_q(\mathfrak {sl}_2)$-module of type 1.
Then $T$ and $T^{\vee}$ are invertible on $V$.
On each
weight space $V(\lambda)$,
\begin{eqnarray*}
&&T^{-1} = 
\sum_{
\genfrac{}{}{0pt}{}{a,b,c \in \mathbb N}{a-b+c=\lambda}
}
\frac{f^a e^b f^c}{
\lbrack a \rbrack^!_q
\lbrack b \rbrack^!_q
\lbrack c \rbrack^!_q}
(-1)^b q^{ac-b},
\\
&&
(T^{\vee})^{-1} = 
\sum_{
\genfrac{}{}{0pt}{}{a,b,c \in \mathbb N}{b-a-c=\lambda}
}
\frac{e^a f^b e^c}{
\lbrack a \rbrack^!_q
\lbrack b \rbrack^!_q
\lbrack c \rbrack^!_q}
(-1)^b q^{ac-b}.
\end{eqnarray*}
\end{lemma}

\noindent We now describe how
$T^{\pm 1}, (T^\vee)^{\pm 1}$ act.

\begin{lemma} \label{lem:TTTTvi}
{\rm (See \cite[Lemma~8.3]{jantzen}.)}
Refer to  the basis $\lbrace v_i\rbrace_{i=0}^d$
of ${\bf V}_d$ from Lemma
\ref{lem:irred}.
For $0 \leq i \leq d$,
\begin{eqnarray*}
&& Tv_i = (-1)^{d-i} q^{(d-i)(i+1)} v_{d-i},
\qquad \qquad
T^\vee v_i = (-1)^{i} q^{i(d-i+1)} v_{d-i},
\\
&&
T^{-1} v_i = (-1)^{i} q^{i(i-d-1)} v_{d-i},
\qquad \qquad
(T^\vee)^{-1} v_i = (-1)^{d-i} q^{(i-d)(i+1)} v_{d-i}.
\end{eqnarray*}
\end{lemma}

\noindent

\begin{lemma}
\label{lem:Tflip}
{\rm (See \cite[Section~8.2]{jantzen}.)}
Let $V$ denote a finite-dimensional 
$U_q(\mathfrak {sl}_2)$-module of type 1. 
Then for $\lambda \in \mathbb Z$,
\begin{eqnarray*}
T V(\lambda) = V(-\lambda), \qquad \qquad 
T^\vee V(\lambda) = V(-\lambda).
\end{eqnarray*}
\end{lemma}

\noindent The operators $T$ and
$T^{\vee}$ are related  in the following way.

\begin{lemma}
\label{lem:TTcomp}
{\rm (See \cite[Lemma~8.4]{jantzen}.)}
Let $V$ denote a finite-dimensional 
$U_q(\mathfrak {sl}_2)$-module of type 1. On each
weight space $V(\lambda)$,
\begin{eqnarray*}
T = (-1)^\lambda q^{\lambda} T^{\vee},
\qquad \qquad
T^{-1} = (-1)^\lambda q^{\lambda} (T^{\vee})^{-1}.
\end{eqnarray*}
\end{lemma}


\noindent The Lusztig automorphism $L$ (resp. $L^{\vee}$)
is related to the Lusztig operator $T$ (resp. $T^{\vee}$)
in the following way.

\begin{lemma}
{\rm (See \cite[Lemma~8.5]{jantzen}.) }
On each finite-dimensional 
$U_q(\mathfrak {sl}_2)$-module of type 1,
\begin{eqnarray*}
L(\xi) = T \xi T^{-1}, \qquad \quad
L^{\vee}(\xi) = T^{\vee} \xi (T^{\vee})^{-1} \qquad \quad
\forall \; \xi \in 
U_q(\mathfrak {sl}_2).
\end{eqnarray*}
\end{lemma}

\section{The equitable presentation
of $U_q(\mathfrak {sl}_2)$}

\noindent In this section we recall the equitable presentation of
$U_q(\mathfrak {sl}_2)$. For more information on this topic
see
\cite{
alnajjar,
bocktingTer,
bidiag,
irt,
qtet,
equit,
nomura,
tersym,
uawe,
fduq,
billiard,
lrt,
wora
}. 
See also
\cite{benkter}.

\begin{lemma}
\label{lem:equit}
{\rm (See \cite[Theorem~2.1]{equit}.)}
$U_q(\mathfrak {sl}_2)$ is isomorphic to the $\mathbb F$-algebra
with generators
$x,y^{\pm 1},z$ and relations
$y y^{-1} = 1$, 
$y^{-1} y = 1$,
\begin{eqnarray}
\frac{q xy -q^{-1} yx}{q-q^{-1}} = 1, 
\qquad
\frac{q yz -q^{-1} zy}{q-q^{-1}} = 1, 
\qquad 
\frac{q zx -q^{-1} xz}{q-q^{-1}} = 1.
\label{eq:r2}
\end{eqnarray}
Given $0 \not=\theta \in \mathbb F$ and $t \in \mathbb Z$, an
isomorphism sends
\begin{eqnarray}
x &\mapsto& k^{-1} - k^{-1-t} e q^{1+t} (q-q^{-1})\theta^{-1},
\label{eq:th1}
\\
y^{\pm 1} &\mapsto & k^{\pm 1},
\label{eq:th2}
\\
z &\mapsto & k^{-1} + f k^t q^{-t} (q-q^{-1})\theta.
\label{eq:th3}
\end{eqnarray}
\noindent The inverse of this isomorphism sends
\begin{eqnarray*}
e &\mapsto & y^t (1-yx)q^{-1-t}(q-q^{-1})^{-1}\theta,
\\
f &\mapsto & (z-y^{-1})y^{-t}q^t(q-q^{-1})^{-1}\theta^{-1},
\\
k^{\pm 1}&\mapsto & y^{\pm 1}.
\end{eqnarray*}
Another isomorphism sends
\begin{eqnarray}
x &\mapsto& k - k^{1+t} f q^{1+t} (q-q^{-1})\theta^{-1},
\label{eq:2th1}
\\
y^{\pm 1} &\mapsto & k^{\mp 1},
\label{eq:2th2}
\\
z &\mapsto & k + e k^{-t} q^{-t} (q-q^{-1})\theta.
\label{eq:2th3}
\end{eqnarray}
\noindent The inverse of this isomorphism sends
\begin{eqnarray*}
 e &\mapsto & (z-y^{-1})y^{-t}q^t (q-q^{-1})^{-1}\theta^{-1},
\\
f &\mapsto & y^t(1-yx)q^{-1-t}(q-q^{-1})^{-1}\theta,
\\
k^{\pm 1}&\mapsto & y^{\mp 1}.
\end{eqnarray*}
\end{lemma}
\begin{proof} In each case,
one checks that the maps in opposite direction are inverse
$\mathbb F$-algebra homomorpisms, and hence
  $\mathbb F$-algebra isomorphisms.
\end{proof}


\begin{definition}
\label{def:xi}
\rm
Pick  $0 \not=\theta \in \mathbb F$ and
$t \in \mathbb Z$.
Under the {\it primary identification} 
(resp. {\it secondary identification})
of type $(\theta, t)$
we identify 
 the algebra $U_q(\mathfrak {sl}_2)$ and the algebra
 given in
 Lemma
\ref{lem:equit},
via the first (resp. second) isomorphism given in
Lemma \ref{lem:equit}.
\end{definition}


\noindent 
The normalized Casimir element $\Lambda$
looks as follows from the equitable point of view.

\begin{lemma} 
\label{lem:lambda6}
Under every identification from
Definition
\ref{def:xi},
 $\Lambda$ is equal to each
of the following:
\begin{eqnarray*}
&&
qx+q^{-1}y+qz-qxyz, \qquad \qquad q^{-1}x + qy+q^{-1}z-q^{-1} zyx,
\\
&&
qy+q^{-1}z+qx-qyzx, \qquad \qquad q^{-1}y + qz+q^{-1}x-q^{-1} xzy,
\\
&&
qz+q^{-1}x+qy-qzxy, \qquad \qquad q^{-1}z + qx+q^{-1}y-q^{-1} yxz.
\end{eqnarray*}
\end{lemma}
\begin{proof} Consider the expression for $\Lambda$
given in 
(\ref{eq:lambda}). Write this expression in
terms of $x,y,z$ using any identification.
Under a primary identification,
$\Lambda = 
qx+q^{-1}y+qz-qxyz$. Under a
secondary identification,
$\Lambda = 
q^{-1}x+qy+q^{-1}z-q^{-1}zyx$.
The six displayed expressions in the lemma statement
are equal by
\cite[Lemma~2.15]{uawe}.
The result follows.
\end{proof}

\noindent Rearranging 
(\ref{eq:r2}) we obtain
\begin{eqnarray*}
 q(1-xy) = q^{-1}(1-yx),
\qquad 
q(1-yz) = q^{-1}(1-zy),
 \qquad 
q(1-zx) = q^{-1}(1-xz).
\end{eqnarray*}

\begin{definition}
{\rm (See \cite[Definition~3.1]{uawe}.)}
\label{def:nuxnuynuz}
\rm
The elements $\nu_x, \nu_y, \nu_z$ of 
$U_q(\mathfrak {sl}_2)$ are defined as follows:
\begin{eqnarray*}
&&\nu_x =
q(1-yz) = 
q^{-1}(1-zy),
\\
&&\nu_y =
q(1-zx) =
q^{-1}(1-xz),
\\
&&\nu_z =
q(1-xy)
=
q^{-1}(1-yx).
\end{eqnarray*}
\end{definition}

\begin{definition}
\label{def:nxnynz}
\rm
(See \cite[Definition~5.2]{equit}.)
The elements $n_x, n_y, n_z$ of 
$U_q(\mathfrak {sl}_2)$ are defined as follows:
\begin{eqnarray*}
n_x = \frac{\nu_x}{q-q^{-1}},
\qquad \qquad
n_y = \frac{\nu_y}{q-q^{-1}},
\qquad \qquad
n_z = \frac{\nu_z}{q-q^{-1}}.
\end{eqnarray*}
\end{definition}

\begin{lemma} 
\label{lem:efnxnz}
Pick $0 \not=\theta \in \mathbb F$ and
$t \in \mathbb Z$.
Under the primary identification of type $(\theta,t)$,
\begin{eqnarray*}
&&
e = \theta q^{-t} y^t n_z, \qquad \qquad
f = - \theta^{-1} q^{1+t} n_x y^{-1-t},
\\
&&
n_z =  \theta^{-1} q^t k^{-t}e,
\qquad \qquad 
n_x =  - \theta   q^{-1-t}fk^{1+t}.
\end{eqnarray*}
Under the  secondary identification of type $(\theta,t)$,
\begin{eqnarray*}
&&
f = \theta q^{-t}y^t n_z, \qquad \qquad
e = - \theta^{-1} q^{1+t} n_x y^{-1-t},
\\
&&
n_z = \theta^{-1} q^t k^t f,
\qquad \qquad 
n_x = -\theta q^{-1-t}ek^{-1-t}.
\end{eqnarray*}
\end{lemma}
\begin{proof} Use
Definition
\ref{def:xi}
and Definitions
\ref{def:nuxnuynuz},
\ref{def:nxnynz}.
\end{proof}

\begin{lemma} 
\label{lem:Ugen}
{\rm (See \cite[Lemma~6.4]{uawe}.)}
The $\mathbb F$-algebra 
$U_q(\mathfrak {sl}_2)$ is generated by
$\nu_x, y^{\pm 1}, \nu_z$ and also by
$n_x, y^{\pm 1}, n_z$.
Moreover
\begin{eqnarray*}
x = y^{-1} - q^{-1} \nu_z y^{-1},
\qquad \qquad 
z = y^{-1} - q^{-1}  y^{-1} \nu_x.
\end{eqnarray*}
\end{lemma}

\begin{lemma} 
\label{lem:znx}
{\rm (See \cite[Lemma~5.4]{equit}.)}
The following hold in 
$U_q(\mathfrak {sl}_2)$:
\begin{eqnarray*}
&&
x n_y = q^2 n_y x,
\qquad \qquad y n_z = q^2 n_z y,
\qquad \qquad z n_x = q^2 n_x z, 
\\
&&
x n_z = q^{-2} n_z x,
 \qquad \qquad y n_x = q^{-2} n_x y,
 \qquad \qquad z n_y = q^{-2} n_y z.
\end{eqnarray*}
\end{lemma}

\noindent In 
\cite{fduq} we described the finite-dimensional 
irreducible 
$U_q(\mathfrak {sl}_2)$-modules from the equitable point of
view. Below we summarize a few points.

\begin{lemma}
\label{lem:uibasis}
{\rm (See \cite[Theorem~10.12]{fduq}.)}
The  
$U_q(\mathfrak {sl}_2)$-module 
${\bf V}_{d,\varepsilon}$
has a basis $\lbrace u_i \rbrace_{i=0}^d$ with respect to which the matrices
representing $x,y,z$ are:
\begin{eqnarray*}
&&
x:\quad  
\varepsilon 
\left(
\begin{array}{ c c c c c c }
 q^{-d} & q^d-q^{-d}   &   &&   & {\bf 0}  \\
   & q^{2-d}  &  q^d-q^{2-d}  &&  &      \\
      &   &  q^{4-d}   & \cdot &&  \\
           &   &  & \cdot  & \cdot & \\
	          &  & &  & \cdot & q^d-q^{d-2} \\
		          {\bf 0}  &&  & &   &  q^d  \\
			          \end{array}
				          \right),
\\
\\
&& y: \quad
\varepsilon \;{\rm diag} (q^d, q^{d-2}, q^{d-4}, \ldots, q^{-d}),
\\
\\
&& z: \quad
\varepsilon 
\left(
\begin{array}{ c c c c c c }
 q^{-d} &    &   &&   & {\bf 0}  \\
  q^{-d}-q^{2-d} & q^{2-d}  &    &&  &      \\
      & q^{-d}-q^{4-d}  &  q^{4-d}   &  &&  \\
           &   &  \cdot & \cdot  &  & \\
	          &  & & \cdot & \cdot &  \\
		          {\bf 0}  &&  & &  q^{-d}-q^d  &  q^d  \\
			          \end{array}
				          \right).
\end{eqnarray*}
\end{lemma}


\begin{lemma}
\label{lem:xyzMF}
For each of $x,y,z$ the action on
${\bf V}_{d,\varepsilon}$  is multiplicity-free, with
eigenvalues
$\lbrace \varepsilon q^{d-2i} \;|\; 0 \leq i \leq d \rbrace$.
\end{lemma}
\begin{proof} By Lemma
\ref{lem:uibasis}.
\end{proof}


\begin{lemma} 
\label{lem:uimore}
{\rm (See \cite[Theorem~11.7]{fduq}.)}
Referring to  Lemma
\ref{lem:uibasis},
with respect to the basis
$\lbrace u_i \rbrace_{i=0}^d$ 
the matrices representing
$n_x, n_z$ are
\begin{eqnarray*}
&&
n_x :\quad
\left(
\begin{array}{ c c c c c c }
 0 &    &   &&   & {\bf 0}  \\
  \lbrack 1 \rbrack_q & 0  &    &&  &      \\
      & q^{-1}\lbrack 2 \rbrack_q  &  0   &  &&  \\
           &   &  \cdot & \cdot  &  & \\
	          &  & & \cdot & \cdot &  \\
		          {\bf 0}  &&  & &  q^{1-d}\lbrack d \rbrack_q &  0  \\
			          \end{array}
				          \right),
\\
\\
&&
n_z: \quad
\left(
\begin{array}{ c c c c c c }
 0 & -q^{d-1}\lbrack d \rbrack_q   &   &&   & {\bf 0}  \\
   & 0  & -q^{d-2} \lbrack d-1 \rbrack_q  &&  &      \\
      &   &  0  & \cdot &&  \\
           &   &  & \cdot  & \cdot & \\
	          &  & &  & \cdot & -\lbrack 1 \rbrack_q\\
		          {\bf 0}  &&  & &   &  0  \\
			          \end{array}
				          \right).
\end{eqnarray*}
\end{lemma}

\begin{lemma}
\label{lem:nxnynzNIL}
{\rm (See \cite[Lemma~8.2]{fduq}.)}
On the 
$U_q(\mathfrak {sl}_2)$-module
${\bf V}_{d,\varepsilon}$, each of
$n^r_x$,
$n^r_y$,
$n^r_z$ is nonzero for
$0 \leq r \leq d$, and
\begin{eqnarray*}
n^{d+1}_x = 0, \qquad \qquad
n^{d+1}_y = 0, \qquad \qquad
n^{d+1}_z = 0.
\end{eqnarray*}
\end{lemma}

\noindent From the equitable point of view we now
 consider finite-dimensional 
$U_q(\mathfrak {sl}_2)$-modules that are not necessarily
irreducible.

\begin{lemma}
\label{lem:yactionIdent}
Let $V$ denote a finite-dimensional
$U_q(\mathfrak {sl}_2)$-module of type 1.
Then $y$ acts on each weight space $V(\lambda)$
as a scalar multiple of the identity. The scalar is
$q^{\lambda}$ (under a primary identification) and
$q^{-\lambda}$ (under a secondary second identification).
\end{lemma}
\begin{proof} By Definition
\ref{def:WS} and
(\ref{eq:th2}),
(\ref{eq:2th2}).
\end{proof}

\begin{lemma}
\label{lem:nxNIL}
Each of $n_x, n_y, n_z$ is nilpotent
on 
finite-dimensional
$U_q(\mathfrak {sl}_2)$-modules.
\end{lemma}
\begin{proof}
Suppose the result is false, and let
the finite-dimensional 
$U_q(\mathfrak {sl}_2)$-module $V$ be
a counterexample with minimal dimension.
By assumption $V\not=0$.
Also, $V$ is
not irreducible by 
  Lemma
\ref{lem:nxnynzNIL}
and the last assertion of Lemma
\ref{lem:irred}.
Therefore $V$ contains
a $U_q(\mathfrak {sl}_2)$-submodule $W$ 
such that $0 \not=W \not=V$.
Since the counterexample $V$ has minimal dimension,
each of $n_x,n_y,n_z$ is nilpotent on the
$U_q(\mathfrak {sl}_2)$-modules
$W$ and  $V/W$.
Consequently each of $n_x,n_y,n_z$ is nilpotent on $V$, for a contradiction.
The result follows.
\end{proof}

\noindent By
Lemma \ref{lem:equit}
the element $y$ is invertible in 
$U_q(\mathfrak {sl}_2)$.
By \cite[Section~3]{equit} the elements $x,z$
are not invertible in 
$U_q(\mathfrak {sl}_2)$. However by
\cite[Corollary~4.5]{equit} the elements $x,z$
are invertible on nonzero finite-dimensional 
$U_q(\mathfrak {sl}_2)$-modules, provided that
${\rm Char}(\mathbb F) \not=2$. We now show
that 
$x,z$
are invertible on nonzero finite-dimensional 
$U_q(\mathfrak {sl}_2)$-modules, without any assumption
about $\mathbb F$.

\begin{lemma}
\label{lem:xyzinv}
The elements $x$ and $z$ are invertible on
nonzero finite-dimensional 
$U_q(\mathfrak {sl}_2)$-modules.
\end{lemma}
\begin{proof}
The proof is similar to the proof of
Lemma
\ref{lem:nxNIL}.
Suppose the result is false, and let
the nonzero finite-dimensional 
$U_q(\mathfrak {sl}_2)$-module $V$ be
a counterexample with minimal dimension.
Then $V$ is
not irreducible by 
  Lemma
\ref{lem:xyzMF}
and the last assertion of Lemma
\ref{lem:irred}.
Therefore $V$ contains
a $U_q(\mathfrak {sl}_2)$-submodule $W$ 
such that $0 \not=W \not=V$.
Since the counterexample $V$ has minimal dimension,
$x$ and $z$ are invertible on the
$U_q(\mathfrak {sl}_2)$-modules
$W$ and  $V/W$.
Consequently $x$ and $z$ are invertible on $V$, for a contradiction.
The result follows.
\end{proof}

\noindent We now describe the Lusztig automorphisms
$L$ and $L^\vee$ from the equitable point of view.

\begin{proposition}
\label{prop:LusztigEquit}
The following 
{\rm (i), (ii)} hold 
for all $0 \not=\theta \in \mathbb F$ and
$t \in \mathbb Z$.
\begin{enumerate}
\item[\rm (i)]
Under the primary (resp. secondary) identification of type $(\theta,t)$
the automorphism 
$ L$ (resp. $L^\vee$) sends
\begin{eqnarray*}
n_x \mapsto \theta^2 q^{-1} y^{-1} n_z y^{-1},
\qquad \qquad 
y \mapsto y^{-1},
\qquad \qquad 
n_z \mapsto  \theta^{-2} q n_x
\end{eqnarray*}
and $L^{-1}$ (resp. 
$(L^\vee)^{-1}$)
sends
\begin{eqnarray*}
n_x \mapsto  \theta^2 q^{-1}n_z,
\qquad \qquad 
y \mapsto y^{-1},
\qquad \qquad 
n_z \mapsto \theta^{-2} q y^{-1}n_x y^{-1}.
\end{eqnarray*}
\item[\rm (ii)] 
Under the primary (resp. secondary) identification of type $(\theta,t)$
the automorphism 
$ L^\vee$ (resp. $L$) sends
\begin{eqnarray*}
n_x \mapsto \theta^2 q y^{-1} n_z y^{-1},
\qquad \qquad 
y \mapsto y^{-1},
\qquad \qquad 
n_z \mapsto  \theta^{-2} q^{-1} n_x
\end{eqnarray*}
and $(L^\vee)^{-1}$ (resp. 
$L^{-1}$)
sends
\begin{eqnarray*}
n_x \mapsto  \theta^2 q n_z,
\qquad \qquad 
y \mapsto y^{-1},
\qquad \qquad 
n_z \mapsto \theta^{-2} q^{-1} y^{-1}n_x y^{-1}.
\end{eqnarray*}
\end{enumerate}
\end{proposition}
\begin{proof} Use
Definition \ref{def:LusztigAut}
and Lemmas
\ref{lem:Linv},
\ref{lem:efnxnz}.
\end{proof}

\noindent We point out two special cases
of Proposition
\ref{prop:LusztigEquit}.

\begin{corollary}
\label{cor:goodtheta}
The following
{\rm (i), (ii)} hold
for all $0 \not=\theta \in \mathbb F$ and
$t \in \mathbb Z$.
\begin{enumerate}
\item[\rm (i)] 
Assume that $\theta^2 = q$.
Under the primary (resp. secondary) identification of type $(\theta,t)$
the automorphism 
$ L$ (resp. $L^\vee$) sends
\begin{eqnarray*}
n_x \mapsto  y^{-1} n_z y^{-1},
\qquad \qquad 
y \mapsto y^{-1},
\qquad \qquad 
n_z \mapsto  n_x
\end{eqnarray*}
and $L^{-1}$ (resp. 
$(L^\vee)^{-1}$)
sends
\begin{eqnarray*}
n_x \mapsto  n_z,
\qquad \qquad 
y \mapsto y^{-1},
\qquad \qquad 
n_z \mapsto  y^{-1}n_x y^{-1}.
\end{eqnarray*}
\item[\rm (ii)] 
Assume that $\theta^2 = q^{-1}$. 
Under the primary (resp. secondary) identification of type $(\theta,t)$
the automorphism 
$ L^\vee$ (resp. $L$) sends
\begin{eqnarray*}
n_x \mapsto  y^{-1} n_z y^{-1},
\qquad \qquad 
y \mapsto y^{-1},
\qquad \qquad 
n_z \mapsto  n_x
\end{eqnarray*}
and $(L^\vee)^{-1}$ (resp. 
$L^{-1}$)
sends
\begin{eqnarray*}
n_x \mapsto n_z,
\qquad \qquad 
y \mapsto y^{-1},
\qquad \qquad 
n_z \mapsto  y^{-1}n_x y^{-1}.
\end{eqnarray*}
\end{enumerate}
\end{corollary}

\noindent Shortly we will describe the
Lusztig operators $T$ and $T^\vee$ from
the equitable point of view. In this description
we will use the $q$-exponential function.

\section{The $q$-exponential function}

\noindent We will be discussing the elements $n_x, n_y, n_z$ of
$U_q({\mathfrak {sl}}_2)$ from Definition
\ref{def:nxnynz}. In this section we recall the $q$-exponential function,
and investigate
\begin{eqnarray}
{\rm exp}_q(n_x), \qquad \qquad 
{\rm exp}_q(n_y), \qquad \qquad 
{\rm exp}_q(n_z).
\label{eq: expqnnn}
\end{eqnarray}

\begin{definition}\rm
\label{def:expq}
(See \cite[p.~204]{tanisaki}.)
Let $V$ denote a vector space over $\mathbb F$ with
finite positive dimension.
Let
$\varphi \in {\rm End}(V)$ be nilpotent. 
Define
\begin{eqnarray}
{\rm exp}_q(\varphi) = \sum_{i\in \mathbb N} \frac{q^{\binom{i}{2}}}
{\lbrack i \rbrack^!_q}
\varphi^i.
\label{eq:EXP}
\end{eqnarray}
\end{definition}

\noindent The following result is well known and readily verified.
\begin{lemma} 
\label{lem:expqInv}
{\rm (See \cite[p.~204]{tanisaki}.)}
Referring to Definition
\ref{def:expq}, the map
${\rm exp}_q(\varphi )$ is invertible; its inverse is
\begin{eqnarray*}
{\rm exp}_{q^{-1}}(-\varphi) =
\sum_{i\in \mathbb N}\frac{(-1)^i q^{-\binom{i}{2}}}{\lbrack i \rbrack^!_q}
\varphi^i.
\end{eqnarray*}
\end{lemma}

\noindent We will make use of the following identity.
\begin{lemma}
\label{lem:expqForm}
Referring to Definition
\ref{def:expq}, we have
\begin{eqnarray}
{\rm exp}_q(q^2 \varphi) \,
\bigl(1-(q^2-1)\varphi \bigr)
= {\rm exp}_q(\varphi).
\label{eq:ExpID}
\end{eqnarray}
\end{lemma}
\begin{proof}
To verify (\ref{eq:ExpID}), express each side
as a power series in $\varphi$ using
(\ref{eq:EXP}), and for $i \in \mathbb N$
compare the coefficients of $\varphi^i$.
\end{proof}

\noindent We now consider
the objects (\ref{eq: expqnnn}). 
In view of Lemma
\ref{lem:nxNIL}, 
we interpret
these objects
as operators
 that act on
 nonzero finite-dimensional 
$U_q({\mathfrak {sl}}_2)$-modules.
As we investigate these operators
we will display some results for
${\rm exp}_q(n_z)$;
similar results hold for
${\rm exp}_q(n_x)$ and
${\rm exp}_q(n_y)$.
\begin{lemma}
\label{lem:obvious}
On nonzero finite-dimensional 
$U_q({\mathfrak {sl}}_2)$-modules,
\begin{enumerate}
\item[\rm (i)]
$
{\rm exp}_q(n_z)^{-1}
\,\Lambda \,
{\rm exp}_q(n_z)
= \Lambda$;
\item[\rm (ii)] 
${\rm exp}_q(n_z)^{-1} \, n_z \, 
{\rm exp}_q(n_z)
= n_z$.
\end{enumerate}
\end{lemma}
\begin{proof} (i) 
The central element $\Lambda$ 
commutes with $n_z$.
Therefore
$\Lambda$ commutes with
${\rm exp}_q(n_z)$ in view of
Definition
\ref{def:expq}.
\\
\noindent (ii) By Definition
\ref{def:expq}.
\end{proof}

\begin{lemma}
\label{lem:yNz}
On nonzero finite-dimensional 
$U_q({\mathfrak {sl}}_2)$-modules,
\begin{eqnarray*}
{\rm exp}_q(n_z)^{-1}
\,y\,
{\rm exp}_q(n_z)
= x^{-1}.
\end{eqnarray*}
\end{lemma}
\begin{proof} Setting $\varphi=n_z$ in  Lemma
\ref{lem:expqForm}, we obtain
${\rm exp}_q(q^2 n_z) \,
\bigl(1-(q^2-1)n_z \bigr)
= {\rm exp}_q(n_z)$.
By Lemma 
\ref{lem:znx}
we have
$y n_z y^{-1} = q^2 n_z$, so 
 $ y \,{\rm exp}_q(n_z) \,y^{-1} = 
  {\rm exp}_q( q^2 n_z)$.
By Definitions
\ref{def:nuxnuynuz},
\ref{def:nxnynz}
we have $n_z=q^{-1}(1-yx)(q-q^{-1})^{-1}$, so
$yx=1-(q^2-1)n_z$.
By these comments
 $ y \,{\rm exp}_q( n_z) \, x =
 {\rm exp}_q( n_z)$.
The result follows. 
\end{proof}

\begin{lemma}
\label{lem:rec}
For an integer $i\geq 1$,
\begin{eqnarray}
z n^i_z - n^i_z z = q^{1-i} \lbrack i \rbrack_q (n^{i-1}_z x - y n^{i-1}_z).
\label{eq:znz}
\end{eqnarray}
\end{lemma}
\begin{proof} 
Using Lemma
\ref{lem:lambda6} and Definition
\ref{def:nuxnuynuz} we obtain
$z \nu_z = \Lambda -q^{-1}x - qy$ and
$ \nu_z z = \Lambda -qx - q^{-1} y$.
Also $\nu_z = (q-q^{-1}) n_z $  by Definition
\ref{def:nxnynz}. By these comments 
$zn_z - n_z z = x-y$. Use this, Lemma
\ref{lem:znx},
and 
induction on $i$ to obtain
(\ref{eq:znz}).
\end{proof}

\begin{lemma}
\label{lem:yzToxz}
On nonzero finite-dimensional 
$U_q({\mathfrak {sl}}_2)$-modules,
\begin{eqnarray}
(y+z)\,{\rm exp}_q(n_z)
= {\rm exp}_q(n_z)\,
 (x+z).
\label{eq:yzToxz}
\end{eqnarray}
\end{lemma}
\begin{proof} To verify
(\ref{eq:yzToxz}),
evaluate each side
using Definition
\ref{def:expq}
and simplify the result using
Lemma
\ref{lem:rec}.
\end{proof}

\begin{lemma} 
\label{lem:zNz}
On nonzero finite-dimensional 
$U_q({\mathfrak {sl}}_2)$-modules,
\begin{eqnarray}
{\rm exp}_q(n_z)^{-1}
\,z\,
{\rm exp}_q(n_z)
= x-x^{-1}+z.
\label{eq:zNz}
\end{eqnarray}
\end{lemma} 
\begin{proof} The left-hand side 
of
(\ref{eq:zNz})
is equal to
\begin{eqnarray*}
{\rm exp}_q(n_z)^{-1} \,(y+z)\, 
{\rm exp}_q(n_z)  
-
{\rm exp}_q(n_z)^{-1} \,y \,
{\rm exp}_q(n_z),
\end{eqnarray*}
which is equal to 
$x+z-x^{-1}$
by Lemmas
\ref{lem:yNz},
\ref{lem:yzToxz}.
\end{proof}

\begin{lemma} 
\label{lem:xNz}
On nonzero finite-dimensional 
$U_q({\mathfrak {sl}}_2)$-modules,
\begin{eqnarray*}
{\rm exp}_q(n_z)^{-1}
\,x\,{\rm exp}_q(n_z)
 = xyx.
\end{eqnarray*}
\end{lemma}
\begin{proof} In the equation 
${\rm exp}_q(n_z)^{-1} \,n_z \,
{\rm exp}_q(n_z) = n_z$,
use $n_z = q(1-xy)(q-q^{-1})^{-1}$ to obtain
${\rm exp}_q(n_z)^{-1} \,(1-xy) \,
{\rm exp}_q(n_z) = 1-xy$.
Evaluate this equation using
Lemma
\ref{lem:yNz} to get the result.
\end{proof}

\begin{lemma}
\label{lem:nxNz}
On nonzero finite-dimensional 
$U_q({\mathfrak {sl}}_2)$-modules,
\begin{eqnarray}
\label{eq:nxNz}
{\rm exp}_q(n_z)^{-1} 
\,n_x \,
{\rm exp}_q(n_z) = x^{-1} n_y x^{-1}.
\end{eqnarray}
\end{lemma}
\begin{proof} To verify
(\ref{eq:nxNz}), eliminate
$n_x, n_y$ using
\begin{eqnarray*}
n_x = q(1-yz)(q-q^{-1})^{-1},
\qquad \qquad 
n_y = q(1-zx)(q-q^{-1})^{-1},
\end{eqnarray*}
and evaluate
the result using
Lemmas
\ref{lem:yNz},
\ref{lem:zNz}.
\end{proof}

\begin{lemma}
\label{lem:nyNz}
On nonzero finite-dimensional 
$U_q({\mathfrak {sl}}_2)$-modules,
\begin{eqnarray}
\label{eq:nyNz}
{\rm exp}_q(n_z)^{-1} 
\,n_y\, 
{\rm exp}_q(n_z)
=
\frac{\Lambda x} {q-q^{-1}} + n_y - \frac{q+q^{-1}}{q-q^{-1}} x^2 + 
x n_z x.
\end{eqnarray}
\end{lemma}
\begin{proof}
Using  Lemma
\ref{lem:lambda6} and Definition
\ref{def:nuxnuynuz} we obtain
$\nu_y y = \Lambda - q z -q^{-1} x$. In this equation,
multiply each side on the left and right by
${\rm exp}_q(n_z)^{-1} $ and
${\rm exp}_q(n_z)$, respectively, and evaluate
the result using Definitions
\ref{def:nuxnuynuz},
\ref{def:nxnynz} and Lemmas
\ref{lem:yNz},
\ref{lem:zNz},
\ref{lem:xNz}.
\end{proof}

\section{Rotators}

In this section we discuss the rotators for
a finite-dimensional
$U_q({\mathfrak {sl}}_2)$-module of type 1.
We comment on the history.
In \cite{equit} a rotator
was constructed using the $q$-exponential function.
In \cite[Section~16]{fduq} we
 found the matrices that
represent this rotator with respect to various bases
for the underlying vector space. In 
\cite[Section~22]{lrt} we investigated the rotator concept 
in a more general setting.
In the present section we will follow
\cite{equit} more or less, adopting
a different point of view,
and giving new proofs that we find more illuminating.

\begin{definition}
\label{def:ROT}
\rm
Let $V$ denote a finite-dimensional 
$U_q({\mathfrak {sl}}_2)$-module of type 1.
By a {\it rotator on $V$} we mean an invertible
$R \in {\rm End}(V)$ such that on $V$,
\begin{eqnarray}
R^{-1} x R = y,
\qquad \qquad
R^{-1} y R = z,
\qquad \qquad
R^{-1} z R = x.
\label{eq:ROT}
\end{eqnarray}
\end{definition}

\begin{lemma} 
Let $V$ denote a finite-dimensional 
$U_q({\mathfrak {sl}}_2)$-module of type 1.
Then there exists a rotator on $V$.
\end{lemma}
\begin{proof}
By Definition
\ref{def:typeONE}
we may assume that
the $U_q({\mathfrak {sl}}_2)$-module $V$ is irreducible.
Define $x',y',z'$ in ${\rm End}(V)$ as follows:
\bigskip

\centerline{
\begin{tabular}[t]{c|ccc}
{\rm element} & $x'$ & $y'$ & $z'$
   \\  \hline
{\rm action on $V$} &
$y$ & $z$  & $x$
     \end{tabular}
}
\bigskip
\noindent The map $y'$ is invertible by
Lemma
\ref{lem:xyzinv}.
The maps $x',y',z'$ satisfy the defining relations
for 
 $U_q({\mathfrak {sl}}_2)$; therefore
$V$  becomes a
 $U_q({\mathfrak {sl}}_2)$-module
such that 
\bigskip

\centerline{
\begin{tabular}[t]{c|ccc}
{\rm element} & $x$ & $y$ & $z$
   \\  \hline
{\rm action on $V$} &
$x'$ & $y'$  & $z'$
     \end{tabular}
}
\bigskip

\noindent 
The new 
 $U_q({\mathfrak {sl}}_2)$-module $V$ is irreducible by
 construction, and
type 1 by
Lemma
\ref{lem:xyzMF}.
Therefore the 
 new $U_q({\mathfrak {sl}}_2)$-module $V$ is
 isomorphic to the  original $U_q({\mathfrak {sl}}_2)$-module $V$.
Let $R \in {\rm End}(V)$ denote an isomorphism of
 $U_q({\mathfrak {sl}}_2)$-modules from the
new 
 $U_q({\mathfrak {sl}}_2)$-module $V$ to the 
 original
 $U_q({\mathfrak {sl}}_2)$-module $V$.
By construction $R$ is invertible and satisfies
(\ref{eq:ROT}). Therefore $R$ is a rotator on $V$.
\end{proof}

\noindent Later in the paper we will construct an operator that
acts as a rotator on the finite-dimensional
 $U_q({\mathfrak {sl}}_2)$-modules of type 1.
For the time being, we focus on the irreducible case.

\begin{lemma}
\label{lem:RotUnique}
Let $R$ denote a rotator on 
the irreducible
 $U_q({\mathfrak {sl}}_2)$-module
$V = {\bf V}_{d}$.
 Then for
$R' \in {\rm End}(V)$
the following are equivalent:
\begin{enumerate}
\item[\rm (i)] $R'$ is a rotator on $V$;
\item[\rm (ii)] there exists
$0 \not=\alpha \in \mathbb F$
such that $R'=\alpha R$.
\end{enumerate}
\end{lemma}
\begin{proof}
${\rm (i)}\Rightarrow {\rm (ii)}$
The composition $R'R^{-1}$ commutes with the 
actions
$x,y,z$  on $V$. By this and Lemma
\ref{lem:uibasis},
there exists $0 \not=\alpha \in \mathbb F$
such that $R'R^{-1}=\alpha I$.
Consequently $R'=\alpha R$.
\\
\noindent ${\rm (ii)}\Rightarrow {\rm (i)}$ Clear.
\end{proof}

\begin{definition}
\label{def:XYZ}
\rm
Consider the irreducible
 $U_q({\mathfrak {sl}}_2)$-module
$V = {\bf V}_{d}$.
Define $X,Y,Z$ in ${\rm End}(V)$ as follows.
By Lemma
\ref{lem:xyzMF},
each of $x,y,z$ is multiplicity-free on 
$V$ with eigenvalues
$\lbrace  q^{d-2i} \;|\;0 \leq i \leq d\rbrace$.
For $0 \leq i \leq d$, $X$ (resp. $Y$)
(resp. $Z$) acts on the
$(q^{d-2i})$-eigenspace of $x$ (resp. $y$)
(resp. $z$) 
as $q^{2i(d-i)}$ times
the identity. By construction $X,Y,Z$ are invertible.
\end{definition}

\noindent We have some comments about Definition
\ref{def:XYZ}. Since the scalars 
$\lbrace q^{d-2i}\rbrace_{i=0}^d$  are mutually distinct,
there exists a polynomial $G=G_d$
in one variable, that has all coefficients in $\mathbb F$ 
and degree at most $d$, such that
\begin{eqnarray*}
G(q^{d-2i}) = q^{2i(d-i)} \qquad \qquad (0 \leq i \leq d).
\end{eqnarray*}
In the above line, replace $i$ by $d-i$ to obtain
\begin{eqnarray*}
G(q^{2i-d}) = q^{2i(d-i)} \qquad \qquad (0 \leq i \leq d).
\end{eqnarray*}
By construction
 the following hold on the $U_q({\mathfrak {sl}}_2)$-module
${\bf V}_{d}$:
\begin{eqnarray}
&&
X = G(x) = G(x^{-1}),
\label{eq:XG}
\\
&&
Y = G(y) = G(y^{-1}),
\label{eq:YG}
\\
&&
Z = G(z) = G(z^{-1}).
\label{eq:ZG}
\end{eqnarray}
\noindent We call $G$ the {\it standard polynomial}
for ${\bf V}_d$.
We now give some results for $Y$; similar results hold
for $X,Z$.

\begin{lemma} On each finite-dimensional irreducible
 $U_q({\mathfrak {sl}}_2)$-module of type 1,
$Yy = yY$ and 
\begin{eqnarray}
&&
Y^{-1} n_x Y = y^{-1} n_x y^{-1}, \qquad \qquad  
Y n_x Y^{-1} = y n_x y,
\label{eq:Y2}
\\
&&Y^{-1} n_z Y = y n_z y, \qquad \qquad  
Y n_z Y^{-1} = y^{-1} n_z y^{-1}. 
\label{eq:Y1}
\end{eqnarray}
\end{lemma}
\begin{proof} Let $V$ denote the module
in question. So $V={\bf V}_d$, where $d$ is
the diameter of $V$.
To verify the given equations, consider the
matrices 
that represent each side with respect to
the basis $\lbrace u_i \rbrace_{i=0}^d$ of $V$ from
 Lemma
\ref{lem:uibasis}. With respect to 
$\lbrace u_i \rbrace_{i=0}^d$
the matrix representing $y$ is given in 
 Lemma
\ref{lem:uibasis}, and the matrices representing
$n_x,n_z$ are given in Lemma
\ref{lem:uimore}. With respect to 
$\lbrace u_i \rbrace_{i=0}^d$
the matrix representing $Y$ is
diagonal, with $(i,i)$-entry $q^{2i(d-i)}$ for
$0 \leq i \leq d$. The results follow
from these comments after a brief calculation.
\end{proof}

\noindent Recall the primary  and secondary identifications from
Definition \ref{def:xi}.
\begin{lemma}
\label{lem:Ymeaning}
Let $V$ denote a finite-dimensional irreducible 
 $U_q({\mathfrak {sl}}_2)$-module of type 1.
Under any identification, $Y$ acts on each weight space
$V(\lambda)$
as a scalar multiple of the identity. The scalar is
$q^{(d^2-\lambda^2)/2}$, where $d$ is the diameter of $V$.
\end{lemma}
\begin{proof} By Lemma
\ref{lem:yactionIdent}
and
Definition
\ref{def:XYZ}, along with
\begin{eqnarray*}
2i(d-i) = \frac{d^2-(d-2i)^2}{2}
\qquad \qquad (0 \leq i \leq d).
\end{eqnarray*}
\end{proof}

\begin{definition}
\label{def:Om}
\rm Let $V$ denote a finite-dimensional irreducible
 $U_q({\mathfrak {sl}}_2)$-module of type 1. Define
 $\Omega \in {\rm End}(V)$ by
 \begin{eqnarray}
\label{eq:OmDef}
\Omega = 
{\rm exp}_q (n_x) \,Y \,
{\rm exp}_q (n_z),
\end{eqnarray}
where $Y$ is from Definition
\ref{def:XYZ}.
\end{definition}

\begin{proposition} 
\label{prop:OmExists}
Let $V$ denote a finite-dimensional irreducible
 $U_q({\mathfrak {sl}}_2)$-module of type 1. Then
the map $\Omega$ from Definition
\ref{def:Om}
is a rotator on $V$.
\end{proposition}
\begin{proof}
We verify that $\Omega$ satisfies the conditions of
Definition
\ref{def:ROT}.
Note that
$\Omega$ is invertible.
We show that 
$\Omega^{-1} x \Omega = y$.
We have
\begin{eqnarray*}
{\rm exp}_q(n_x)^{-1}
\,x \,
{\rm exp}_q(n_x) &=& x+y-y^{-1}
\\
&=& y-q^{-1} \nu_z y^{-1}
\end{eqnarray*}
and
\begin{eqnarray*}
Y^{-1} (y-q^{-1} \nu_z y^{-1}) Y &=& y - q^{-1}Y^{-1} \nu_z Y y^{-1}
\\
&=& y(1-q^{-1} \nu_z)
\end{eqnarray*}
and
\begin{eqnarray*}
{\rm exp}_q(n_z)^{-1}
\,y(1-q^{-1}\nu_z) \,
{\rm exp}_q(n_z) &=& x^{-1} (1-q^{-1}\nu_z) 
\\
&=& y.
\end{eqnarray*}
\noindent By these comments and
(\ref{eq:OmDef}) we find
$\Omega^{-1} x \Omega = y$. 
Next we show that $\Omega^{-1} y \Omega = z$.
We have
\begin{eqnarray*}
{\rm exp}_q(n_x)^{-1}
\,y \,
{\rm exp}_q(n_x) &=& yzy
\\
&=& (1-q^{-1} \nu_x)y
\end{eqnarray*}
and
\begin{eqnarray*}
Y^{-1} (1-q^{-1} \nu_x)y  Y &=& y - q^{-1}Y^{-1} \nu_x Y y
\\
&=& y-q^{-1}  y^{-1}\nu_x
\\
&=& y-y^{-1}+z 
\end{eqnarray*}
and
\begin{eqnarray*}
{\rm exp}_q(n_z)^{-1}
\,(y-y^{-1} + z) \,
{\rm exp}_q(n_z) &=& 
x^{-1} - x+  x - x^{-1}+z
\\
&=& z.
\end{eqnarray*}
\noindent By these comments and
(\ref{eq:OmDef}) we find
$\Omega^{-1} y \Omega = z$. 
Next we show that $\Omega^{-1} z \Omega = x$.
We have
\begin{eqnarray*}
&&
{\rm exp}_q(n_x)^{-1}
\,z \,
{\rm exp}_q(n_x) = y^{-1},
\\
&&
Y^{-1} y^{-1} Y = y^{-1},
\\
&&
{\rm exp}_q(n_z)^{-1}
\,y^{-1} \,
{\rm exp}_q(n_z) = x. 
\end{eqnarray*}
By these comments and (\ref{eq:OmDef}) we find
$\Omega^{-1} z \Omega = x$. 
We have shown that $\Omega$ is a rotator on $V$.
\end{proof}

\begin{definition}\rm
Let $V$ denote a finite-dimensional irreducible
 $U_q({\mathfrak {sl}}_2)$-module of type 1. By
the {\it standard rotator on $V$} we mean
the rotator $\Omega $ from 
 Definition
\ref{def:Om}.
\end{definition}

\begin{lemma}
\label{lem:stRot1}
On each finite-dimensional irreducible
 $U_q({\mathfrak {sl}}_2)$-module of type 1,
the standard rotator $\Omega$ satisfies
\begin{eqnarray}
\Omega^{-1} n_x \Omega = n_y,
\qquad \qquad 
\Omega^{-1} n_y \Omega = n_z,
\qquad \qquad 
\Omega^{-1} n_z \Omega = n_x
\label{eq:Omnxnynz}
\end{eqnarray}
and also
\begin{eqnarray}
&&\Omega^{-1} \, {\rm exp}_q(n_x) \, \Omega = {\rm exp}_q(n_y),
\label{eq:OmExp1}
\\
&&
\Omega^{-1} \, {\rm exp}_q(n_y) \, \Omega = {\rm exp}_q(n_z),
\label{eq:OmExp2}
\\
&&\Omega^{-1} \, {\rm exp}_q(n_z) \, \Omega = {\rm exp}_q(n_x).
\label{eq:OmExp3}
\end{eqnarray}
\end{lemma}
\begin{proof} Use Definition
\ref{def:ROT}.
\end{proof}

\begin{lemma}
\label{lem:stRot2}
On each finite-dimensional irreducible
 $U_q({\mathfrak {sl}}_2)$-module of type 1,
the standard rotator $\Omega$ satisfies
\begin{eqnarray}
\Omega^{-1} X \Omega = Y, \qquad \qquad 
\Omega^{-1} Y \Omega = Z, \qquad \qquad 
\Omega^{-1} Z \Omega = X.
\label{eq:OmXYZ}
\end{eqnarray}
\end{lemma}
\begin{proof} Let $V$ denote the
 $U_q({\mathfrak {sl}}_2)$-module in question,
 and let $G$ denote the standard polynomial for $V$.
We show that
$\Omega^{-1} X \Omega = Y$ holds on $V$.
The equation
$\Omega^{-1} x \Omega = y$ holds on $V$.
By this and
(\ref{eq:XG}),
(\ref{eq:YG})   we find that on $V$,
\begin{eqnarray*}
\Omega^{-1} X \Omega =
\Omega^{-1} G(x) \Omega 
=
G(\Omega^{-1} x \Omega) 
=
G(y) 
= Y.
\end{eqnarray*}
The remaining equations in
(\ref{eq:OmXYZ}) are similarly obtained.
\end{proof}

\begin{proposition}
\label{lem:Om3}
On every finite-dimensional irreducible
 $U_q({\mathfrak {sl}}_2)$-module of type 1,
the standard rotator $\Omega$ is equal to each of
\begin{eqnarray}
{\rm exp}_q (n_x) \,Y \,
{\rm exp}_q (n_z),
\qquad 
{\rm exp}_q (n_y) \,Z \,
{\rm exp}_q (n_x),
\qquad
{\rm exp}_q (n_z) \,X \,
{\rm exp}_q (n_y).
\end{eqnarray}
\end{proposition}
\begin{proof} In the equation 
(\ref{eq:OmDef}), multiply each side on the
left and right by $\Omega^{-1}$ and $\Omega$, respectively.
Evaluate the result using
(\ref{eq:OmExp1})--(\ref{eq:OmExp3})
and
Lemma
\ref{lem:stRot2}.
\end{proof}

\begin{note}
\label{note:connect}
\rm 
The operator ${\cal R}$ from
\cite[Definition~17.8]{fduq} is the same thing as
our $\Omega^{-1}$.
Also, the operator
called $\Omega$ in 
\cite[Definition~7.4]{equit} is the same thing as our
$\Omega q^{-d^2/2}$ (if $d$ is even) and
$\Omega q^{(1-d^2)/2}$ (if $d$ is odd). Here $d$ denotes
the diameter of the 
$U_q({\mathfrak {sl}}_2)$-module in question.
\end{note}

\begin{lemma}
\label{lem:Omega3}
On each finite-dimensional irreducible $U_q({\mathfrak {sl}}_2)$-module 
of type 1, the map 
$\Omega^3$ acts as a scalar multiple of the identity.
The scalar is $(-1)^d q^{d(d-1)}$, where $d$ is the diameter of the module.
\end{lemma}
\begin{proof}
By \cite[Lemma~16.5]{fduq}
or \cite[Corollary~8.5]{equit},
together with Note
\ref{note:connect}.
\end{proof}

\noindent We mention a result for later use.

\begin{lemma}
\label{lem:nzXY}
On each finite-dimensional irreducible
 $U_q({\mathfrak {sl}}_2)$-module of type 1,
\begin{eqnarray*}
{\rm exp}_q (n_z)^{-1} \,Y \,
{\rm exp}_q (n_z) = X.
\end{eqnarray*}
\end{lemma}
\begin{proof}
Let $V$ denote the 
 $U_q({\mathfrak {sl}}_2)$-module in question,
  and let $G$ denote the
standard polynomial for $V$.
By Lemma
\ref{lem:yNz} 
the following holds on $V$:
\begin{eqnarray*}
{\rm exp}_q (n_z)^{-1} \, y \,
{\rm exp}_q (n_z) =  x^{-1}.
\end{eqnarray*}
Therefore on $V$,
\begin{eqnarray*}
{\rm exp}_q (n_z)^{-1} \,G( y) \,
{\rm exp}_q (n_z) = G(x^{-1}).
\end{eqnarray*}
The result follows in view of
(\ref{eq:XG}),
(\ref{eq:YG}).
\end{proof}

\noindent In Proposition 
\ref{lem:Om3}
we gave three formulae for the standard rotator. We now
give additional formulae for this rotator.

\begin{proposition}
\label{prop:OmegaSixForm}
On every finite-dimensional irreducible 
 $U_q({\mathfrak {sl}}_2)$-module of type 1,
the standard rotator $\Omega$ is equal to each of
\begin{eqnarray}
&&
{\rm exp}_q (n_x) \,
{\rm exp}_q (n_z) \,X,
\qquad 
{\rm exp}_q (n_y) \,
{\rm exp}_q (n_x) \,Y,
\qquad 
{\rm exp}_q (n_z) \,
{\rm exp}_q (n_y) \,Z,
\label{eq:OmR1}
\\
&&
Z\,{\rm exp}_q (n_x) \,
{\rm exp}_q (n_z), 
\qquad 
X\,{\rm exp}_q (n_y) \,
{\rm exp}_q (n_x),
\qquad 
Y\, {\rm exp}_q (n_z) \,
{\rm exp}_q (n_y).
\qquad 
\label{eq:OmR2}
\end{eqnarray}
\end{proposition}
\begin{proof} By 
Proposition
\ref{lem:Om3}
and
Lemma \ref{lem:nzXY}.
\end{proof}

\section{The maps $\tau_x, \tau_y,\tau_z$ and the
Lusztig operators $T, T^\vee$ }

\noindent Let $V$ denote a finite-dimensional irreducible
 $U_q({\mathfrak {sl}}_2)$-module of type 1, with
 standard rotator $\Omega$. In this section
 we use $\Omega$ to define three elements in ${\rm End}(V)$,
 denoted $\tau_x, \tau_y, \tau_z$. We discuss how 
  $\tau_x, \tau_y, \tau_z$
  are related to the maps $X,Y,Z$ from
 Definition
\ref{def:XYZ}.
  We then show how $\tau_y$
 is related to the Lusztig operators $T$ and $T^\vee$.

\begin{definition} \rm
\label{def:txtytz}
Let $V$ denote a finite-dimensional irreducible
 $U_q({\mathfrak {sl}}_2)$-module  of type 1.
Define $\tau_x, \tau_y, \tau_z$ in
${\rm End}(V)$ by
\begin{eqnarray*}
\tau_x = 
{\rm exp}_q (n_y) \,\Omega,
\qquad \qquad
\tau_y = 
{\rm exp}_q (n_z) \,\Omega,
\qquad \qquad
\tau_z = 
{\rm exp}_q (n_x) \,\Omega.
\end{eqnarray*}
\end{definition}

\noindent We have some comments about
Definition \ref{def:txtytz}.

\begin{lemma} 
On each finite-dimensional irreducible 
 $U_q({\mathfrak {sl}}_2)$-module  of type 1, the maps
$\tau_x,  
\tau_y,  
\tau_z $ are invertible. 
\end{lemma} 

\begin{lemma}
\label{lem:tauAlt}
On each finite-dimensional irreducible 
 $U_q({\mathfrak {sl}}_2)$-module  of type 1,
\begin{eqnarray*}
\tau_x = 
\Omega \,{\rm exp}_q (n_z),
\qquad \qquad
\tau_y = 
\Omega \,{\rm exp}_q (n_x),
\qquad \qquad
\tau_z = 
\Omega \,{\rm exp}_q (n_y).
\end{eqnarray*}
\end{lemma}
\begin{proof} By
(\ref{eq:OmExp1})--(\ref{eq:OmExp3})
and
Definition
\ref{def:txtytz}.
\end{proof}

\noindent Next we give some results for $\tau_y$; similar
results hold for $\tau_x, \tau_z$.

\begin{lemma}
\label{prop:tauConj}
On each finite-dimensional irreducible 
 $U_q({\mathfrak {sl}}_2)$-module of type 1,
\begin{enumerate}
\item[\rm (i)] $\tau^{-1}_y n_x \tau_y = y^{-1} n_z y^{-1}$;
\item[\rm (ii)] $\tau^{-1}_y y \tau_y = y^{-1}$;
\item[\rm (iii)] $\tau^{-1}_y n_z \tau_y = n_x$.
\end{enumerate}
\end{lemma}
\begin{proof} Let $V$ denote the 
 $U_q({\mathfrak {sl}}_2)$-module in question.
\\
\noindent (i) 
Using Lemma \ref{lem:nxNz} we find that on $V$,
\begin{eqnarray*}
\tau^{-1}_y n_x \tau_y =
\Omega^{-1} 
\,{\rm exp}_q (n_z)^{-1} \,
 n_x 
\, 
{\rm exp}_q (n_z) \, \Omega
= 
\Omega^{-1} 
x^{-1} n_y x^{-1} 
 \Omega
= y^{-1} n_z y^{-1}.
\end{eqnarray*}
\noindent (ii)
Using Lemma \ref{lem:yNz} we find that on $V$,
\begin{eqnarray*}
\tau^{-1}_y y \tau_y =
\Omega^{-1} 
\,{\rm exp}_q (n_z)^{-1} \,
 y 
\, 
{\rm exp}_q (n_z) \, \Omega
= 
\Omega^{-1} 
x^{-1}  
 \Omega
= y^{-1}.
\end{eqnarray*}
\noindent (iii) 
Using Lemma \ref{lem:obvious}(ii) we find that on $V$,
\begin{eqnarray*}
\tau^{-1}_y n_z \tau_y =
\Omega^{-1} 
\,{\rm exp}_q (n_z)^{-1} \,
 n_z 
\, 
{\rm exp}_q (n_z) \, \Omega
= 
\Omega^{-1} 
n_z  
 \Omega
= n_x.
\end{eqnarray*}
\end{proof}

\begin{lemma} On each finite-dimensional irreducible
 $U_q({\mathfrak {sl}}_2)$-module of type 1,
\begin{eqnarray}
X = \Omega^3 \tau^{-2}_x,
\qquad \qquad 
Y = \Omega^3 \tau^{-2}_y,
\qquad \qquad 
Z = \Omega^3 \tau^{-2}_z.
\label{eq:Om3}
\end{eqnarray}
\end{lemma}
\begin{proof} We verify the middle equation.
We have
\begin{eqnarray}
\label{eq:Yty}
Y \tau^2_y = 
Y \, {\rm exp}_q (n_z) \,
\Omega \, 
 {\rm exp}_q (n_z) \, \Omega.
\end{eqnarray}
Consider the right-hand side of
(\ref{eq:Yty}).
We have
$Y \, {\rm exp}_q (n_z) = 
{\rm exp}_q (n_z) \, X $ by Lemma
\ref{lem:nzXY}, and 
$\Omega\, {\rm exp}_q (n_z) = 
{\rm exp}_q (n_y) \,\Omega $ by 
(\ref{eq:OmExp2}).
Also 
${\rm exp}_q (n_z) \,X \, 
{\rm exp}_q (n_y) = \Omega$ by Lemma
\ref{lem:Om3}. By these comments, the right-hand side of
(\ref{eq:Yty}) 
is equal to $\Omega^3$. The result follows.
\end{proof}

\noindent Recall the primary and secondary identifications
from Definition
\ref{def:xi}.

\begin{lemma} \label{lem:tauAction}
Let $V$ denote a finite-dimensional irreducible
$U_q({\mathfrak {sl}}_2)$-module of type 1. 
Under any identification, 
$\tau_y V(\lambda) = V(-\lambda)$
for all $\lambda \in \mathbb Z$.
\end{lemma}
\begin{proof} 
Use Lemma
\ref{lem:yactionIdent} and
Lemma \ref{prop:tauConj}(ii).
\end{proof}

\begin{proposition} 
\label{prop:tauLu}
Pick  $0 \not=\theta \in \mathbb F$ and 
$t \in \mathbb Z$.
\begin{enumerate}
\item[\rm (i)] Assume that $\theta^2=q$.
Under the primary identification of type
$(\theta,t)$ 
the
following holds on 
the
$U_q({\mathfrak {sl}}_2)$-module
${\bf V}_d$:
\begin{eqnarray}
\tau_y = (-1)^d  \theta^{d^2} T^{-1}.
\label{eq:tauT}
\end{eqnarray}
\item[\rm (ii)] Assume that $\theta^2=q$.
Under the secondary identification of type 
$(\theta,t)$ the
following holds on 
the
$U_q({\mathfrak {sl}}_2)$-module
${\bf V}_d$:
\begin{eqnarray}
\tau_y = (-1)^d \theta^{d^2} (T^{\vee})^{-1}.
\label{eq:tauTvee}
\end{eqnarray}
\item[\rm (iii)] Assume that $\theta^2=q^{-1}$.
Under the primary identification of type
$(\theta,t)$ 
the
following holds on 
the
$U_q({\mathfrak {sl}}_2)$-module
${\bf V}_d$:
\begin{eqnarray}
\tau_y = \theta^{-d^2} (T^\vee)^{-1}.
\label{eq:tauT3}
\end{eqnarray}
\item[\rm (iv)] Assume that $\theta^2=q^{-1}$.
Under the secondary identification of type 
$(\theta,t)$ the
following holds on 
the
$U_q({\mathfrak {sl}}_2)$-module
${\bf V}_d$:
\begin{eqnarray}
\tau_y = \theta^{-d^2} T^{-1}.
\label{eq:tauTvee4}
\end{eqnarray}
\end{enumerate}
\end{proposition}
\begin{proof}
For the time being, assume any identification.
On $V= {\bf V}_d$,
\begin{eqnarray*}
\tau_y
&=& \Omega \,{\rm exp}_q(n_x)
\\
&=& 
{\rm exp}_q(n_x)\,Y\, 
{\rm exp}_q(n_z)
\, {\rm exp}_q(n_x)
\\
&=& 
{\rm exp}_q(n_x)\,Y\, 
{\rm exp}_q(n_z) Y^{-1} Y
\, {\rm exp}_q(n_x) Y^{-1} Y
\\
&=& 
{\rm exp}_q(n_x)\, 
{\rm exp}_q(Yn_zY^{-1})
\, {\rm exp}_q(Yn_xY^{-1}) Y
\\
&=& {\rm exp}_q(n_x)\, 
{\rm exp}_q(y^{-1}n_zy^{-1})
\, {\rm exp}_q(yn_x y) Y.
\end{eqnarray*}
Consider the action of $\tau_y$ on a weight space
$V(\lambda)$. By construction and Lemma
\ref{lem:Ymeaning}
we find that on $V(\lambda)$,
\begin{equation}
\tau_y = \sum_{a,b,c \in \mathbb N}
\frac{
q^{\binom{a}{2}}
q^{\binom{b}{2}}
q^{\binom{c}{2}}
}{
\lbrack a \rbrack^!_q
\lbrack b \rbrack^!_q
\lbrack c \rbrack^!_q}
n^a_x (y^{-1}n_z y^{-1})^b (yn_x y)^c q^{(d^2-\lambda^2)/2}. 
\label{eq:tauExpand2}
\end{equation}
\noindent (i)
We show that
(\ref{eq:tauT}) holds on $V$.
Pick $a,b,c \in \mathbb N$
and consider the corresponding summand in
(\ref{eq:tauExpand2}).
We now write this summand in terms of
 $e,f,k$. Using the primary identification of type $(\theta,t)$ in Lemma
\ref{lem:efnxnz}, along with $ke=q^2ek$ and $kf=q^{-2}fk$, we obtain
\begin{eqnarray*}
&&n^a_x = (-\theta q^{-1-t}f k^{1+t})^a =
(-1)^a \theta^a q^{-a^2(1+t)} f^a k^{a(1+t)},
\\
&&(y^{-1} n_z y^{-1})^b 
=
(\theta^{-1} q^t k^{-1-t} e k^{-1})^b 
= \theta^{-b} q^{-b^2(2+t)}  e^b k^{-b(2+t)},
\\
&&(y n_x y)^c =
(- \theta q^{-1-t} k f k^{2+t})^c = 
(-1)^c \theta^c q^{-c^2(3+t)} f^c k^{c(3+t)},
\\
&&
f^a k^{a(1+t)} e^b k^{-b(2+t)} f^c k^{c(3+t)}
= 
f^a e^b f^c k^{a(1+t)-b(2+t)+c(3+t)} q^{2ab(1+t)-2ac(1+t)+
2bc(2+t)}.
\end{eqnarray*}
Observe that on $V(\lambda)$,
\begin{eqnarray*}
k^{a(1+t)-b(2+t)+c(3+t)}
 = 
q^{\lambda (a(1+t)-b(2+t)+c(3+t))}.
\end{eqnarray*}
By the above comments, for $a,b,c$ the corresponding summand
in 
(\ref{eq:tauExpand2}) acts on $V(\lambda)$ as
a scalar multiple of
$f^a e^b f^c$.
By Lemma
\ref{lem:efmove} we have 
$f^a e^b f^c V(\lambda)\subseteq V(\mu)$, where
$\lambda-\mu= 2(a-b+c)$. Note that $\mu=-\lambda$ if and only 
if $a-b+c=\lambda$. By this and
Lemma \ref{lem:tauAction},
the equation
(\ref{eq:tauExpand2}) remains valid if we restrict
the sum to those $a,b,c \in \mathbb N$ such that
$a-b+c=\lambda$.
Evaluating (\ref{eq:tauExpand2})  using the above
discussion 
we find that
 on $V(\lambda)$,
\begin{eqnarray}
\label{eq:compare}
\tau_y = 
(-1)^d \theta^{d^2}
\sum_{
\genfrac{}{}{0pt}{}{a,b,c \in \mathbb N}{a-b+c=\lambda}
}
\frac{f^a e^b f^c}{
\lbrack a \rbrack^!_q
\lbrack b \rbrack^!_q
\lbrack c \rbrack^!_q}
(-1)^b q^{ac-b}.
\end{eqnarray}
By (\ref{eq:compare})
and Lemma
\ref{lem:TTcheckInv},
we find 
that
$\tau_y = 
(-1)^d \theta^{d^2} T^{-1}$ on
$V(\lambda)$. Since $V$ is spanned by its weight spaces,
$\tau_y = 
(-1)^d \theta^{d^2} T^{-1}$ on
$V$.
\\
\noindent (ii)
We show that 
(\ref{eq:tauTvee}) 
holds on $V$.
Recall that $\tau_y$ acts on each weight space $V(\lambda)$ as
in (\ref{eq:tauExpand2}).
Pick $a,b,c \in \mathbb N$
and consider the corresponding summand in
(\ref{eq:tauExpand2}).
We now write this summand in terms of
 $e,f,k$. Using the secondary identification of type $(\theta,t)$ 
 in Lemma
\ref{lem:efnxnz}, along with $ke=q^2ek$ and $kf=q^{-2}fk$, we obtain
\begin{eqnarray*}
&&n^a_x = (-\theta q^{-1-t} e k^{-1-t})^a =
(-1)^a \theta^a
q^{-a^2(1+t)}
e^a k^{-a(1+t)},
\\
&&(y^{-1} n_z y^{-1})^b 
=
(\theta^{-1} q^t k^{1+t} f k)^b 
= \theta^{-b} q^{-b^2(2+t)}  f^b k^{b(2+t)},
\\
&&(y n_x y)^c =
(- \theta q^{-1-t} k^{-1} e k^{-2-t})^c = 
(-1)^c \theta^c q^{-c^2(3+t)} e^c k^{-c(3+t)},
\\
&&
e^a k^{-a(1+t)} f^b k^{b(2+t)} e^c k^{-c(3+t)}
= 
e^a f^b e^c k^{-a(1+t)+b(2+t)-c(3+t)} q^{2ab(1+t)-2ac(1+t)
+2bc(2+t)}.
\end{eqnarray*}
On $V(\lambda)$,
\begin{eqnarray*}
k^{-a(1+t)+b(2+t)-c(3+t)}
 = 
q^{\lambda (-a(1+t)+b(2+t)-c(3+t))}.
\end{eqnarray*}
This time around, 
(\ref{eq:tauExpand2}) remains valid if we restrict
the sum to those $a,b,c \in \mathbb N$ such that
$b-a-c=\lambda$.
Evaluating 
(\ref{eq:tauExpand2}) using these comments
 we find that
 on $V(\lambda)$,
\begin{eqnarray}
\label{eq:compare2}
\tau_y = 
(-1)^d \theta^{d^2}
\sum_{
\genfrac{}{}{0pt}{}{a,b,c \in \mathbb N}{b-a-c=\lambda}
}
\frac{e^a f^b e^c}{
\lbrack a \rbrack^!_q
\lbrack b \rbrack^!_q
\lbrack c \rbrack^!_q}
(-1)^b q^{ac-b}.
\end{eqnarray}
By (\ref{eq:compare2})
and Lemma
\ref{lem:TTcheckInv},
we find 
that
$\tau_y = 
(-1)^d \theta^{d^2} (T^\vee)^{-1}$ on
$V(\lambda)$. Since $V$ is spanned by its weight spaces,
$\tau_y = 
(-1)^d \theta^{d^2} (T^\vee)^{-1}$ on
$V$.
\\
\noindent (iii)
Proceeding as in the proof of (i) above, we find that on
each weight space $V(\lambda)$,
\begin{eqnarray}
\tau_y = 
(-1)^\lambda q^{-\lambda} \theta^{-d^2}
\sum_{
\genfrac{}{}{0pt}{}{a,b,c \in \mathbb N}{a-b+c=\lambda}
}
\frac{f^a e^b f^c}{
\lbrack a \rbrack^!_q
\lbrack b \rbrack^!_q
\lbrack c \rbrack^!_q}
(-1)^b q^{ac-b}.
\label{eq:tau3}
\end{eqnarray}
Evaluate 
(\ref{eq:tau3}) 
using Lemmas
\ref{lem:TTcheckInv},
\ref{lem:TTcomp}.
\\
\noindent (iv) Proceeding as in the proof of (ii) above,
we find that on each weight space $V(\lambda)$,
\begin{eqnarray}
\label{eq:tau4} 
\tau_y = 
(-1)^\lambda q^{\lambda} \theta^{-d^2}
\sum_{
\genfrac{}{}{0pt}{}{a,b,c \in \mathbb N}{b-a-c=\lambda}
}
\frac{e^a f^b e^c}{
\lbrack a \rbrack^!_q
\lbrack b \rbrack^!_q
\lbrack c \rbrack^!_q}
(-1)^b q^{ac-b}.
\end{eqnarray}
Evaluate 
(\ref{eq:tau4})
using Lemmas
\ref{lem:TTcheckInv},
\ref{lem:TTcomp}.
\end{proof}

\section{The rotator $\mathfrak{R}$ and 
the Lusztig operators $T,T^\vee$}

\noindent 
In this section we introduce an operator $\mathfrak R$
that acts as a rotator on each finite-dimensional
 $U_q({\mathfrak {sl}}_2)$-module of type 1.
We show how $\mathfrak R$ is related to the Lusztig operators
$T$, $T^\vee$.
\medskip

\noindent Until the end of this section, the following
notation and assumptions are in effect.
Assume that $\mathbb F$ contains the square root of $q$.
Pick $\theta \in \mathbb F$
such that $\theta^2 = q$ or $\theta^2 = q^{-1}$.
Pick $t \in \mathbb Z$.
Assume the primary or secondary identification
of type $(\theta,t)$, as in Definition
\ref{def:xi}. Define
\begin{eqnarray}
\label{eq:Sqrtq}
q^{1/2} = 
\begin{cases}
-\theta &  {\mbox{\rm if $\;\theta^2=q$}}; \\
\theta^{-1} & {\mbox{\rm if $\;\theta^2=q^{-1}$}}.
\end{cases}
\end{eqnarray}

\begin{lemma}
\label{lem:Sqrtq}
For $d \in \mathbb N$,
\begin{eqnarray*}
q^{d^2/2} = 
\begin{cases}
(-1)^d \theta^{d^2} &  {\mbox{\rm if $\;\theta^2=q$}}; \\
\theta^{-d^2} & {\mbox{\rm if $\;\theta^2=q^{-1}$}}.
\end{cases}
\end{eqnarray*}
\end{lemma}
\begin{proof} The integers $d$ and $d^2$ have the 
same parity, so
$(-1)^d =
(-1)^{d^2}$.
The result
follows from this and
(\ref{eq:Sqrtq}).
\end{proof}

\noindent Next we define an operator $\Upsilon $, that
acts on each type 1 finite-dimensional
$U_q({\mathfrak {sl}}_2)$-module in an
$\mathbb F$-linear fashion. We now give the action.

\begin{definition}
\label{def:Psidef}
\rm
Let $V$ denote a  finite-dimensional
$U_q({\mathfrak {sl}}_2)$-module of type 1.
Then $\Upsilon $ acts on 
each weight
space $V(\lambda)$ 
as a scalar multiple of the identity.
The scalar is 
$ q^{-\lambda^2/2}$, where
$q^{1/2}$ is from 
{\rm (\ref{eq:Sqrtq})}.
\end{definition}

\noindent We have two comments about $\Upsilon $.

\begin{lemma}
\label{lem:WPsiInv}
Let $V$ denote a finite-dimensional
$U_q({\mathfrak {sl}}_2)$-module of type 1. Then
each 
$U_q({\mathfrak {sl}}_2)$-submodule of $V$ is 
$\Upsilon $-invariant.
\end{lemma}
\begin{proof} 
Let $W$ denote the
$U_q({\mathfrak {sl}}_2)$-submodule in question.
By Lemmas
\ref{lem:typeOneSub},
\ref{lem:dsws} we see that
$W$ is spanned by eigenvectors for
  $k$.
By Definition
\ref{def:Psidef}, these vectors are eigenvectors for $\Upsilon $.
Therefore
$W$ is $\Upsilon $-invariant.
\end{proof}

\begin{lemma} 
\label{lem:PsiVSY}
On the
$U_q({\mathfrak {sl}}_2)$-module
${\bf V}_d$,
\begin{eqnarray*}
Y = q^{d^2/2} \Upsilon .
\end{eqnarray*}
\end{lemma}
\begin{proof} 
By
Lemma 
\ref{lem:Ymeaning}
and
Definition \ref{def:Psidef}.
\end{proof}

\noindent We now define an operator
$\mathfrak R$ that acts on each type 1
finite-dimensional 
$U_q({\mathfrak {sl}}_2)$-module in an
$\mathbb F$-linear fashion.

\begin{definition}
\label{def:tOmDef}
\rm On each finite-dimensional
$U_q({\mathfrak {sl}}_2)$-module of type 1,
 \begin{eqnarray}
\label{eq:tOmDef}
{\mathfrak R} = 
{\rm exp}_q (n_x) \,\Upsilon  \,
{\rm exp}_q (n_z),
\end{eqnarray}
where $\Upsilon $ is from Definition
\ref{def:Psidef}.
\end{definition}

\begin{lemma} 
\label{lem:OmVStildeOm}
On the
$U_q({\mathfrak {sl}}_2)$-module
${\bf V}_d$,
\begin{eqnarray*}
{\mathfrak R}= q^{-d^2/2} \Omega,
\end{eqnarray*}
where $q^{1/2}$ is from
{\rm (\ref{eq:Sqrtq})}.
\end{lemma}
\begin{proof} 
Compare 
(\ref{eq:OmDef}) 
and
(\ref{eq:tOmDef}) using
Lemma \ref{lem:PsiVSY}.
\end{proof}

\begin{lemma}
\label{lem:WRinv}
Let $V$ denote a finite-dimensional 
$U_q({\mathfrak {sl}}_2)$-module of type 1. Then each
$U_q({\mathfrak {sl}}_2)$-submodule of $V$ is $\mathfrak R$-invariant.
\end{lemma}
\begin{proof}
Let $W$ denote the 
$U_q({\mathfrak {sl}}_2)$-submodule in question.
By assumption $W$ is invariant under
$n_x$ and $n_z$.
Therefore $W$ is invariant under 
${\rm exp}_q(n_x)$
and ${\rm exp}_q(n_z)$.
By Lemma
\ref{lem:WPsiInv},
$W$ is invariant under $\Upsilon $.
By these comments and Definition
\ref{def:tOmDef},
$W$ is invariant under $\mathfrak R$.
\end{proof}

\begin{lemma}
\label{lem:RisRot}
The operator $\mathfrak R$ acts as a rotator on
each finite-dimensional 
$U_q({\mathfrak {sl}}_2)$-module of type 1.
\end{lemma}
\begin{proof} 
Let $V$ denote the 
$U_q({\mathfrak {sl}}_2)$-module in question.
By Definition
\ref{def:typeONE}, $V$ is a direct sum of
irreducible 
$U_q({\mathfrak {sl}}_2)$-submodules that have type 1.
Each summand $W$ is $\mathfrak R$-invariant
by Lemma
\ref{lem:WRinv}.
By Lemma
\ref{lem:OmVStildeOm},
$\mathfrak R$ acts on $W$ as 
a scalar multiple of $\Omega$.
Now by Lemma
\ref{lem:RotUnique} and
Proposition
\ref{prop:OmExists},
 $\mathfrak R$
acts on $W$ as a rotator.
 Consequently
$\mathfrak R$ acts on $V$ as a rotator.
\end{proof}

\noindent Recall from
 Definition
\ref{def:xi}
the primary and secondary identification
of type $(\theta, t)$.

\begin{theorem}
\label{thm:LOm}
The rotator $\mathfrak R$ is related to the
Lusztig operators $T, T^\vee$
according to the table below:
\bigskip

\centerline{
\begin{tabular}[t]{c|cc}
 & {\rm primary ident. of type $(\theta,t)$} & 
 {\rm secondary ident. of type 
 $(\theta, t)$}
   \\  \hline
$\theta^2=q$ &
$T^{-1} = {\rm exp}_q(n_z)\,\mathfrak R$ 
 &
$(T^\vee)^{-1} = {\rm exp}_q(n_z)\,\mathfrak R$ 
  \\ 
$\theta^2=q^{-1}$ &
$(T^\vee)^{-1} = {\rm exp}_q(n_z)\,\mathfrak R$ 
&
$T^{-1} = {\rm exp}_q(n_z)\,\mathfrak R$ 
   \\
     \end{tabular}
}
\bigskip

\end{theorem}
\begin{proof} First assume  $\theta^2=q$ 
and the primary identification of type $(\theta,t)$.
Let $V$ denote a finite-dimensional 
$U_q({\mathfrak {sl}}_2)$-module of type 1.
We show that
$T^{-1} = {\rm exp}_q(n_z)\,\mathfrak R$ 
on $V$.
By Definition
\ref{def:typeONE}
and Lemma 
\ref{lem:WRinv},
we may assume without loss that the
$U_q({\mathfrak {sl}}_2)$-module $V$ is irreducible.
Let $d$ denote the diameter of $V$. On $V$,
\begin{eqnarray*}
T^{-1} 
&=&  (-1)^d \theta^{-d^2} \tau_y 
\qquad \qquad \qquad \mbox{\rm by Proposition
\ref{prop:tauLu}(i)}
\\
&=&  q^{-d^2/2} \tau_y 
\qquad \qquad \qquad \qquad \mbox{\rm by Lemma
\ref{lem:Sqrtq}
}
\\
&=&  q^{-d^2/2} 
\, {\rm exp}_q(n_z)\,\Omega 
\qquad \qquad \mbox{\rm by Definition
\ref{def:txtytz}
}
\\
&=& 
 {\rm exp}_q(n_z)\,\mathfrak R 
\;\;\qquad \qquad \qquad  \mbox{\rm by Lemma
\ref{lem:OmVStildeOm}.
}
\end{eqnarray*}
For the other cases the proof is similar.
\end{proof}

\section{Acknowledgement} The author thanks Kazumasa Nomura for giving this
paper a close reading and offering valuable suggestions. The author also
thanks the referee for improving the exposition.





\noindent Paul Terwilliger \hfil\break
\noindent Department of Mathematics \hfil\break
\noindent University of Wisconsin \hfil\break
\noindent 480 Lincoln Drive \hfil\break
\noindent Madison, WI 53706-1325 USA \hfil\break
\noindent email: {\tt terwilli@math.wisc.edu} \hfil\break

\end{document}